\documentclass[final]{dcthesis}

\usepackage{amsmath,amssymb,amsthm}
\usepackage{graphicx}
\usepackage{verbatim}

\newcommand{\N}{\mathbb{N}}
\newcommand{\Z}{\mathbb{Z}}
\newcommand{\Q}{\mathbb{Q}}
\newcommand{\R}{\mathbb{R}}
\newcommand{\C}{\mathbb{C}}

\newcommand{\F}{\mathbb{F}}

\newcommand*{\ra}{\rightarrow}

\newtheorem{prop}{Proposition}[chapter]
\newtheorem{theorem}[prop]{Theorem}
\newtheorem{lemma}[prop]{Lemma}
\newtheorem{corollary}[prop]{Corollary}
\newtheorem{proposition}[prop]{Proposition}
\newtheorem*{mainquestion}{Main problem}

\theoremstyle{definition}
\newtheorem{definition}[prop]{Definition}
\theoremstyle{remark}
\newtheorem{remark}[prop]{Remark}
\newtheorem{example}[prop]{Example}

\title{Quantum graphs and their spectra}
\author{Ralf Rueckriemen}
\date{\today}
\field{Mathematics}
\degree{Doctor of Philosophy}
\committee{Carolyn Gordon}{Craig Sutton}{Pete Winkler}{Peter Kuchment}

\begin{document}

\frontmatter

\maketitle

\chapter*{Abstract}
\addcontentsline{toc}{section}{Abstract}
We show that families of leafless quantum graphs that are isospectral for the standard Laplacian 
are finite. We show that the minimum edge length is a spectral invariant. We give an upper bound for the size of isospectral 
families in terms of the total edge length of the quantum graphs.

 We define the Bloch spectrum of a quantum graph to be the map that assigns to each element in the deRham cohomology 
the spectrum of an associated magnetic Schr\"odinger operator. We show that the 
Bloch spectrum determines the Albanese torus, the block structure and the planarity of the graph. It determines a geometric dual of a 
planar graph. This enables us to show that the Bloch spectrum identifies and completely determines planar $3$-connected quantum graphs.

\chapter*{Acknowledgement}
\addcontentsline{toc}{section}{Preface}

First and foremost, I would like to thank my advisor Carolyn Gordon. Without her constant support, this thesis could never have been 
written. She let me freely choose a topic of research and fully supported my decision although the topic that I finally chose, 
quantum graphs, is not at the center of her expertise. 
She always took her time to meet with me. I am especially grateful for the scrutiny with which she looked 
over my results. I do not remember how many times I came to see her all excited about a new `proof' I had found only to be 
shown that my arguments still had gaping holes. She taught me to write rigorous and readable math. Lastly she launched me 
into the academic community, she seems to know everyone, she enabled me to go to numerous conferences and meet all the 
important people in my field.

I would like to thank the administration and the members of this department for their excellent guidance through the program 
and their continuous belief in my abilities. I struggled quite a bit in my first two years here to pass all the qualification 
exams but the department was always willing to give me another chance and urged me to try again until I finally managed 
to fulfill all the requirements. Let me especially point out the department administrators Tracy Molony, Annette Luce and 
Stephanie Kvam and 
the professors Peter Doyle, Carolyn Gordon, Tom Shemanske, Craig Sutton, Dorothy Wallace and Pete Winkler. 

Finally I would like to thank the graduate students of the department. The sense of community and their willingness to help 
was not only much appreciated in the academic realm but also for my social well-being. I would like to list all of them but 
let me mention at least a few by name: Patricia Cahn and Matt Mahony for numerous discussions about my research and math in 
general, Sarah Wright for recommendations about good teaching, and Katie Kinnard for help in writing a job application that 
could actually lead to a job.

\tableofcontents


\listoffigures

\mainmatter


\chapter{Introduction}

The relationship between a space and the spectrum of a differential operator on it is often very 
close. It is studied for manifolds, orbifolds, Riemann surfaces, combinatorial graphs and quantum graphs. The question 
to what degree the spectrum of the Laplace operator determines the underlying space was popularized by Kac in \cite{Kac66} 
in the manifold setting. Numerous examples of isospectrality have been found, see the history for references to some examples.

This opens up a whole array of new questions: Which properties of a space are spectrally determined? 
Which classes of spaces are spectrally determined? How 
can one construct examples of isospectrality? How big can isospectral families be? Trying to answer these and related questions 
is at the heart of the field of spectral geometry. 

In this thesis we will restrict our attention to quantum graphs. They are the most recent object of interest in this field and often 
provide a bridge between manifolds and combinatorial graphs. 

Our first result says that families of leafless isospectral quantum graphs are finite, that is, there can be at most finitely
many leafless quantum graphs that have the same spectrum. 

The main results in this thesis come from considering an entire collection of operators and their spectra on a given quantum graph. 
We show that this collection of spectra determines various properties of the quantum graph, some of which are not determined by 
the spectrum of a single operator. 
For example, it determines the block structure which provides a broad overview of the structure of the graph. It also 
determines whether or not a quantum graph is planar. It completely determines quantum graphs in a certain class of graphs.

We will first give some history and context of the subject and then give a more detailed overview of this thesis.

\section{History}

\subsection{Quantum graphs}

A quantum graph (or metric graph) is a finite combinatorial graph where each edge is equipped with a positive finite length. 
Usually there is a Schr\"odinger operator acting on the graph implicitly understood in the background. Some people 
put the operator as part of the definition of a quantum graph.

A function on a quantum graph consists of a function on each edge, where the edges are viewed as intervals. The Schr\"odinger operator 
acts on the space of all functions that are smooth on each edge and satisfy some suitable boundary conditions at the vertices. 
A Schr\"odinger operator is a second order differential operator on each edge with leading term the standard Laplacian 
$-\left(\frac{\partial}{\partial x}\right)^2$. The first order part is called the magnetic potential.

The spectrum of a Schr\"odinger operator on a quantum graph is real, discrete, infinite, and bounded from below, it has a 
single accumulation point at infinity. The multiplicity of each eigenvalue is finite. This is true in a significantly broader 
context, a proof can be found in \cite{Kuchment04}.

Quantum graphs are studied in mathematics and physics. They first arose in physical chemistry in the 1930s in work by Pauling. 
He used them as models for $\pi$-electron orbitals in conjugate molecules. The atoms in the molecule are the vertices, the bonds 
between the atoms are the paths that the electrons travel on. The movement of the electrons then obeys a Schr\"odinger equation.
Quantum graphs are used to approximate behavior and gain a theoretical understanding 
of objects in mesoscopic physics and nanotechnology. They serve as simplified models in many settings involving 
wave propagation. Mathematicians like to study the spectral theory of quantum graphs because it offers a nice trade-off between the 
 richness of structure for manifolds and the ease of computations of examples for combinatorial graphs.
The survey articles \cite{Kuchment08} and \cite{Post09} provide an excellent introduction and numerous 
references to the literature.

The fact that quantum graphs are essentially $1$-dimensional makes explicit computations possible in various situations. 
Assume the Schr\"odinger operator is of the form $\left(i\frac{\partial}{\partial x}+A(x)\right)^2$ on all edges.
Then the eigenvalues can be found numerically as the zeros of the determinant of a version of the adjacency 
matrix of the quantum graph. The eigenfunctions are all simple sine waves 
on each edge. For this kind of operator we also have algebraic 
relations between the spectrum and the geometry of the quantum graph; that is, there exists an exact trace formula. 
The trace formula is one of the key 
tools in the study of quantum graphs and their spectra. It is a distributional 
equality where the left side of the equation is an infinite sum over all the eigenvalues of the Schr\"odinger operator. 
The right side of the equation contains geometrical information about the quantum graph such as the the total edge length, 
and the Euler characteristic of the graph, as well as an infinite sum over all periodic orbits in the graph.

Trace formulae or asymptotic expansions exist in a variety of settings. The first such formulae were the Selberg trace formula for 
manifolds with negative curvature (\cite{Selberg56}), and the Poisson summation formula for the Laplacian on flat tori. 
Although these two are exact formulas, other trace formulas on manifolds are just asymptotic. The wave trace relates the eigenvalues of 
the Laplacian on a closed surface to the closed geodesics, see \cite{ColindeVerdiere73} and \cite{DuistermaatGuillemin75}.
Quantum graphs, on the other hand, admit exact trace formulae. The first trace 
formula for quantum graphs was proven in \cite{Roth83}. He used heat kernel methods and restricted himself to the standard Laplacian 
and standard boundary conditions. Since then various generalizations have been shown, some based on heat kernel methods,
others on wave kernel techniques, \cite{KPS07}, \cite{KottosSmilansky99}. The trace formula can be interpreted as an 
index theorem for quantum graphs, \cite{FKW07}. The paper 
 \cite{BolteEndres08} gives an introduction and a survey of the various versions of the trace formula. 

The behavior and distribution of the eigenvalues and eigenfunctions has been studied from various perspectives. There are sharp 
lower bounds on the eigenvalues in terms of the total edge length of the quantum graph, \cite{Friedlander05}. 
Nodal domains of eigenfunctions and their relation to isospectrality have been studied 
in \cite{GUW04}, \cite{BSS06}, \cite{Berkolaiko08} and \cite{BOS08}. 

It has been observed that quantum graphs share many properties with quantum chaotic systems, \cite{KottosSmilansky97}, 
\cite{KottosSmilansky99}. It is conjectured that for a suitable 
generic sequence of quantum graphs with increasing number of edges the limiting 
statistical distribution of the eigenvalues coincides with the one for families of random matrices with increasing rank, see 
\cite{GnutzmannSmilansky06} and \cite{Keating08}.

\subsection{Spectral geometry and the problem of isospectrality}


The first examples of isospectral non-isometric manifolds were found by Milnor in \cite{Milnor64}. 
The fact that combinatorial graphs are not determined by their spectrum is even older, \cite{CollatzSinogowitz57}. 
There are copious examples of isospectrality, some references are \cite{SSW06}, \cite{RSW08} for isospectral orbifolds, 
\cite{Vigneras80}, \cite{Buser92} and \cite{BrGoGu98} for Riemann surfaces, and
\cite{vonBelow99}, \cite{GutkinSmilansky01}, \cite{BPB09} for quantum graphs. 

In all that follows we will only consider finite combinatorial graphs, finite quantum graphs with finite edge lengths and 
compact closed oriented manifolds. By Riemann 
surface we mean a compact oriented 2-dimensional Riemannian manifold with constant curvature negative one.



It is well known that certain features of a space are encoded in the spectrum. For instance, in the combinatorial graph 
setting the number of vertices of a combinatorial graph is trivially spectrally 
determined. (It is equal to the number of eigenvalues counting multiplicities.) One can use Weyl asymptotics and the asymptotics of the heat trace 
to show that the dimension,  
volume and total scalar curvature of manifolds, \cite{BGM71}, and the dimension and volume of Riemannian orbifolds, 
\cite{Donnelly79} and \cite{DGGW08}, are 
spectrally determined. Quantum graphs admit a much stronger exact trace formula instead of
just asymptotics. The trace formula directly shows 
that the total edge length and the Euler characteristic of a quantum graph are spectrally determined. Most of the results in this 
thesis come from a careful study of the other terms in the trace formula. 

On the other hand, many properties of a space are not spectrally determined. There are examples of isospectral manifolds with 
different fundamental groups \cite{Vigneras80} and different maximal scalar curvature \cite{GGSWW98}. Orientability of manifolds is not 
spectrally determined either \cite{BerardWebb95}, \cite{MiatelloRossetti01}. Isospectral orbifolds can have different 
isotropy orders, \cite{RSW08}. There are examples of isospectral 
combinatorial and quantum graphs where one is planar and the other one is not, see \cite{CDS95} and \cite{vonBelow99}.

It has been shown that 2-dimensional flat tori \cite{BGM71} and round spheres of dimension up to 6 \cite{Tanno73} are uniquely 
determined by the spectrum of the standard Laplacian among oriented manifolds. Complete combinatorial graphs are spectrally 
determined \cite{Chung97}.  

Constructions of isospectral manifolds are usually through the Sunada method, \cite{Sunada85} and its generalizations 
(see \cite{Gordon09} for a summary), 
or the torus action method, \cite{Gordon01a} or \cite{Schueth01a}. The Sunada method generalizes to various other settings 
including combinatorial and quantum 
graphs, \cite{BPB09}. Isospectral combinatorial graphs can also be found through explicit computations or through Seidel switching, 
\cite{Seress00}. 

There exist smooth isospectral deformations of manifolds, see for example \cite{GordonWilson84} and \cite{Schueth99}, 
so the isospectral families are infinite
in this case. It is conjectured that isospectral families of manifolds are compact; this was first shown for 
plane domains \cite{OPS88b} and closed surfaces \cite{OPS88a}.
Since then it has been shown in some special cases, for example manifolds with absolutely bounded 
sectional curvature \cite{Zhou97}.
On the other hand, for Riemann surfaces of genus $g$ the size of an isospectral family is at most $e^{720g^2}$, 
\cite{Buser92}. Families of isospectral combinatorial graphs have the same number of vertices and thus are trivially finite. We will 
show in this thesis that isospectral families of leafless quantum graphs are finite and prove an upper bound 
dependent on the total edge length of the quantum graph.

\subsection{The Bloch spectrum}

The major theme of this thesis will be a variation of the question of which properties are spectrally determined. Instead of 
looking at the spectrum of a single operator one looks at a whole collection of spectra. One then asks to what degree the 
additional spectra give more information than a single spectrum. One way of doing so is to 
look at the spectrum of the Laplacian on manifolds 
acting on functions and differential forms. This approach has been used to show that the round sphere in all dimensions is 
determined by the spectra of the Laplacian acting on functions and $1$-forms, \cite{Patodi71}. Another idea, the one we are 
going to use,  is to consider 
only operators acting on functions but vary the lower order terms of the differential operator. We will look at the 
spectra of all Schr\"odinger operators of 
the form $(d+2\pi i \alpha)^*(d+2\pi i\alpha)$ where the $1$-form $\alpha$ changes and call this collection of spectra the 
{\bf Bloch spectrum} of the quantum graph. 

The theory of the Bloch spectrum or Bloch theorem also goes by the name of Floquet theory. It was first invented by the 
mathematician Gaston Floquet in the late 19th century, he used it to study certain periodic differential equations. 
It was then reinvented and generalized to higher dimensions by the physicist Felix Bloch in the 1920s. 
He used it to describe the movement of a free electron in a crystalline material. 
The movement of the electron is described by the classical Schr\"odinger equation, modified by a potential that represents the 
atoms in the material. Reduced to the $1$-dimensional case this means we have the Schr\"odinger equation 
$-\frac{\hbar^2}{2m_e}\frac{\partial^2 \psi(x)}{\partial x^2}+V(x)\psi(x)=\lambda \psi(x)$ with a periodic 
potential $V$ and then look for globally bounded solutions. The solutions have the form 
$\psi(x)=\theta_{\lambda}(x)e^{i\sqrt{\lambda} x}$ for some 
periodic function $\theta_{\lambda}$ and are called Bloch waves. The set of $\lambda$ where 
the Schr\"odinger equation has a nontrivial solution is the set of energy levels. 

In mathematical language this can be rephrased as follows. 
The classical Bloch spectrum of a torus $\R^n/L$ assigns to each character $\chi:L\ra \C^*$ the spectrum of the standard Laplacian 
acting on the space of functions on $\R^n$ that satisfy $f(x+l)=\chi(l)f(x)$ for all $l\in L$.   See, for example, \cite{ERT84} 
for inverse spectral results concerning the Bloch spectrum.  As pointed out by Guillemin \cite{Guillemin90}, the Bloch 
spectrum can also be interpreted as the collection of spectra of all operators $\nabla^*\nabla$ acting on 
sections of a trivial bundle, as $\nabla$ varies over all connections with zero curvature. The set of these connections is given by 
$\nabla=(d+i\alpha)$, where $\alpha$ is a harmonic 1-form on $\R^n/L$. 
(One may take $\alpha$ to be any closed 1-form, but the spectrum depends only on the cohomology class of $\alpha$, so one may 
always assume $\alpha$ to be harmonic.) The correspondence with the classical notion is given by the association of the character 
$\chi(l)=e^{2\pi i\alpha(l)}$ to the harmonic form $\alpha$, where now $\alpha$ is viewed as a linear functional on $\R^n$. 
This interpretation of the Bloch spectrum admits a generalization 
to operators acting on sections of an arbitrary Hermitian line bundle over a torus, where now one considers all connections 
with, say, harmonic curvature, see \cite{GGKW08}. 

We will transplant this idea from tori to quantum graphs. Both notions of the Bloch spectrum of a torus 
can be carried over to quantum graphs and we will show that they are equivalent. We will then ask which properties 
of a quantum graph are determined by the Bloch spectrum and which classes of quantum graphs can be identified and 
characterized by their Bloch spectrum.

\section{Overview}

Before we delve into quantum graphs we first need a few basic facts about combinatorial graphs collected in Chapter 2. We will 
define the block structure, which contains the broad structure of the graph. We will then talk about planarity of 
graphs and how to detect it, explain the notion of a dual of 
a planar graph and finally talk about the spectral theory of combinatorial graphs. 
 
In Chapter 3 we will define quantum graphs and then proceed to define a notion of differential forms on quantum graphs. 
We will show that this notion reproduces the expected deRham cohomology and use the exterior derivative $d$ from functions to $1$-forms
and its adjoint to define the Schr\"odinger operators. To make the eigenvalue problem well-defined we need to impose some 
boundary conditions at the vertices. We will impose the Kirchhoff boundary conditions on functions, which require that the 
function be continuous and that the sum of the inward pointing 
derivatives on all edges incident at the vertex be zero. Kirchhoff conditions model a conservation of flow. For $1$-forms 
we use the naturally associated conditions. We want that the differential of a function that satisfies Kirchhoff boundary 
conditions satisfies the boundary conditions for $1$-forms.

We will cite and explain a trace formula in Chapter 4. It is the main tool we use to extract information about the 
quantum graph from the spectrum. 

In Chapter 5 we will consider the following question:

\itshape
\noindent How big can a family of isospectral quantum graphs be? 
\upshape

We will consider the standard Laplacian $\Delta_0=d^*d$ on a leafless quantum graph. We show that the minimum edge 
length\footnote{For technical reasons a loop of length $l$ is counted as two edges of length $l/2$, see \ref{loop_half_count} 
for details.}
is a spectral invariant and normalize it to be $1$. 
It turns out that the size of isospectral families can then be bounded in terms of the total 
edge length of the quantum graph. Using constructions of isospectral combinatorial graphs we show that there exist isospectral 
families of equilateral quantum graphs whose size grows exponentially in the total edge length of the quantum graphs.
We proceed to show the following upper bound:
\begin{theorem}
The size of isospectral families of leafless quantum graphs with minimum edge length $1$ and total edge length $\mathcal{L}$ is 
at most $e^{7\mathcal{L}\ln(\mathcal{L})}$. In particular all isospectral families are finite.
\end{theorem}

We will define the notion of the Albanese torus of a quantum graph in Chapter 6. It is defined as the quotient of the 
real homology by the integer homology of the graph. It inherits an inner product structure that makes it a Riemannian manifold. 
The Albanese torus contains information about the length of the cycles in a quantum graph and about which cycles overlap. 

We will define the notion of Bloch spectrum in Chapter 7. Using our notion of differential forms we define the Bloch 
spectrum as the collection of all spectra of Schr\"odinger operators of the form $\Delta_{\alpha}=(d+2\pi i\alpha)^*(d+2\pi i\alpha)$ and 
vary the $1$-form $\alpha$. Similarly to the setting of flat tori the spectrum 
depends only on the equivalence class of $\alpha$ in $H_{dR}^1(G,\R) \slash H_{dR}^1(G,\Z)$.
We then introduce the Bloch spectrum using characters of the fundamental group. 
We will show the equivalence of the two notions by associating the character $\chi_{\alpha}(\gamma)=e^{-2\pi i\int_{\gamma}\alpha}$ 
to a $1$-form $\alpha$, where $\gamma \in \pi_1(G)$.

It is known (see Chapter 7 for details) that the spectrum of the standard Laplacian $\Delta_0$ determines the dimension $n$ of 
$H^1(G,\R)$.  Thus from that spectrum alone, we know that $H^1(G,\R)/H^1(G,\Z)$ is isomorphic as a torus (i.e., as a Lie group) 
to $\R^n/ \Z^n$. 
Hence we can view the Bloch spectrum as a map that associates a spectrum to each $\alpha\in \R^n/ \Z^n$. 
We have now set up all the machinery to start answering our main question:

\begin{mainquestion}
Suppose we are given a map that assigns a spectrum to each element $\alpha$ of $\R^n / \Z^n$ and we know these spectra form the Bloch 
spectrum of a quantum graph $G$. From this information, can one reconstruct $G$ both 
combinatorially and metrically?
\end{mainquestion}

In Chapter 8 we show that the Bloch spectrum determines the length of a shortest 
representative of each element in $H_1(G,\Z)$, see Theorem \ref{Bloch_H1}. 
We will consider a generic $1$-form $\alpha$, i.e., one whose orbit is dense in the torus $\R^n/\Z^n$, 
and we will just consider the spectra associated to an interval in the orbit of $\alpha \in \R^n/\Z^n$.
Theorem \ref{Bloch_H1} is the main theorem that relates the Bloch spectrum 
to the quantum graph. The other theorems are just consequences from this one. 

We will use the results of Chapter 8 to find properties that are determined by the Bloch spectrum in Chapter 9. 
We show that this information can be used to 
compute the lengths of all cycles in the quantum graph and the lengths of their overlaps. This shows
\begin{theorem}
 The Bloch spectrum determines the Albanese torus, $Alb(G)=H_1(G,\R) \slash H_1(G,\Z)$, of a quantum graph as a Riemannian manifold. 
\end{theorem}
Note that the spectrum of a single Schr\"odinger operator does not determine the  
Albanese torus, there are examples of isospectral quantum graphs with different Albanese tori, \cite{vonBelow99}. If the quantum graph is 
equilateral the Albanese torus also determines the complexity of the graph by a theorem in \cite{KotaniSunada00}.
\begin{theorem}
The Bloch spectrum determines the block structure of a quantum graph (see Definition \ref{define_block}). 
\end{theorem}

\begin{theorem}
The Bloch spectrum determines whether or not a graph is planar. 
\end{theorem}
Planarity is not determined by the spectrum of a single Schr\"odinger operator, \cite{vonBelow99}.
The information about the 
homology we read out from the Bloch spectrum allows us to construct a geometric dual of a planar quantum graph. 
We use it to recover the underlying combinatorial graph of a planar $3$-connected quantum graph from the Bloch spectrum.

In Chapter 10 we show that if we 
know the underlying combinatorial graph and it is $3$-connected and planar then the Bloch spectrum determines 
the length of all edges in the quantum graph, so we can recover the full quantum graph. Together with results from the previous 
chapter this implies:
\begin{theorem}
Planar $3$-connected graphs can be identified and completely reconstructed from their Bloch spectrum.
\end{theorem}

In the last chapter we will treat disconnected graphs and show that our results still hold in this case.

\chapter{Combinatorial graph theory}

This chapter collects various basic facts about combinatorial graphs that will be required later. The material is mostly taken from 
\cite{Diestel05}, which provides an excellent introduction to the area. 

\begin{definition}
 A {\bf directed combinatorial graph} $G$ is a finite set of vertices $V$, a finite set of edges $E$ or $E(G)$, and a function 
$E \ra V\times V$ that associates each edge with its initial and terminal vertices.

A {\bf combinatorial graph} is a directed combinatorial graph where we have `forgotten the directions' of the edges, that is, we consider 
the two edges $(v,v')$ and $(v',v)$ as equivalent.
\end{definition}

This definition of a combinatorial graph allows for loops; that is, edges that start and end at the same vertex, and multiple edges; 
that is, several edges with the same start and end vertex. 

\begin{definition}
A combinatorial graph without loops and multiple edges is called a {\bf simple graph}.
\end{definition}

The {\bf degree of a vertex} is the number of edges that are incident to it. As both ends of a loop are attached to the same vertex, 
adding a loop increases the degree of the vertex by $2$.

\begin{remark}
We will assume throughout that our graphs are connected and in particular do not have isolated vertices, we will treat disconnected 
graphs in Chapter 11. 

 We also assume that there are no vertices of degree $2$. Once we pass to 
quantum graphs, two edges connected 
by a vertex of degree $2$ with Kirchhoff boundary condition behave exactly the same way as a single longer edge does.
\end{remark}

We need various different notions of paths on a graph.
\begin{definition}
Let $G$ be a directed combinatorial graph.
\begin{enumerate}
 \item A {\bf path} in a graph is a finite alternating sequence of edges and vertices starting and ending with a vertex such that 
every edge sits in between its two end vertices.
\item A {\bf closed walk} in a graph is an alternating cyclic sequence of edges and vertices such that every edge sits in between 
its two end vertices.
\item An {\bf oriented cycle} in a graph is a closed walk that does not repeat any edges or vertices. 
\item A {\bf cycle} is a subgraph that consists of a set of edges and vertices that forms an oriented cycle.
\end{enumerate}

\end{definition}

Whenever we use the word cycle we mean it in this graph theoretical sense and not in a homological sense.
A graph without any cycles is called a {\bf tree}.

\begin{definition}
\label{overlap}
Let $\gamma_1$ and $\gamma_2$ be two oriented cycles  in a graph. 
We say they have {\bf edges of positive overlap} if they have an edge in common and pass through it in the same direction. 
We say they have {\bf edges of negative overlap} if they have an edge in common and pass through it in opposite directions.
\end{definition}

Note that two oriented cycles can have both edges of positive and negative overlap.

\begin{lemma}
\label{basis_of_cycles}
 Every graph admits a basis of its homology that consists of oriented cycles. 
\end{lemma}
\begin{proof}
 Pick a spanning tree of the graph. Associate to each edge of $G$ not in the spanning tree the oriented cycle that consists of 
this edge and the (unique) path in the spanning tree that connects its end points. This collection of oriented cycles is a 
basis of the homology. 
\end{proof}

\begin{definition}
 We call a graph with no leaves, that is, vertices of degree $1$, a {\bf leafless} graph. 
\end{definition}

\section{Connectivity and the block structure of a graph}

The spectrum of the Laplacian not only determines the number of connected components of a combinatorial or quantum graph 
but it also contains more subtle information about the 
connectivity. We will show later that the Bloch spectrum determines the block structure, which provides a broad 
view of the structure of a graph. Here, we will introduce some of the language related to connectivity of graphs as well 
as the definition of the block structure. 

\begin{definition}
\label{k-connected}
 A graph $G$ is called {\bf $k$-connected} if any two vertices $v,v' \in V$ can be connected by $k$ disjoint paths. The paths are 
called disjoint if they do not share any edges or vertices (apart from $v$ and $v'$).
\end{definition}

\begin{definition}
 A vertex $v$ in $G$ is called a {\bf cut vertex} if $G \setminus \{v\}$ is disconnected.
\end{definition}

\begin{definition}
Consider the set of all cycles in the graph. Declare two cycles equivalent if they have at least one edge in common. This 
generates an equivalence relation. 
A {\bf block} is the union of all cycles in one equivalence class.
\end{definition}

Note that two cycles can be equivalent even if they do not share an edge.

\begin{remark}
This is a slight deviation from the standard definition. It is changed to 
allow graphs with loops and multiple edges. 
Edges that are not part of any cycle are not part of any block in our 
definition. Usually these edges are counted as blocks, too.
\end{remark}

Note that all loops in the graph are blocks; all other blocks are $2$-connected.

\begin{definition}
\label{define_block}
We define the {\bf block structure} of a graph as follows. Each block in the graph is replaced by a small circle that we call a fat vertex. 
The cut vertices contained in this block correspond to the different attaching points on the fat vertex. For loops we interpret 
their vertex as the cut vertex where they are attached to the rest of the graph. 

All other blocks or 
remaining edges sharing one of the cut vertices with the original block are connected at the respective attaching point on the fat vertex. 

It does not matter how the different attaching points are arranged around the fat vertex.
We explicitly allow several fat vertices to be directly connected to each other without an edge in 
between.
\end{definition}

\begin{remark}
Again this is a nonstandard definition. Our definition contains the same information about the graph as the standard 
one modulo the addition of loops and multiple edges.
\end{remark}

\begin{example}
Figure \ref{qgraph_blockstructure} shows an example of a graph and its block structure. For simplicity of recognition all 
blocks in the graphs are either loops or copies of the complete graph on $4$ vertices, $K_4$.

\begin{figure}[ht]
\centering
\scalebox{0.8}{\includegraphics{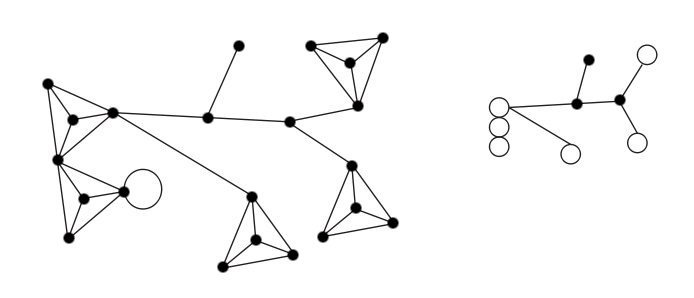}}
\caption{A quantum graph and its block structure}
\label{qgraph_blockstructure}
\end{figure}
\end{example}

\begin{remark}
\label{cycle_single_block}
Any cycle in the graph is confined to a single block.
Thus the vertices and edges in the block structure never form a cycle and the block structure has a tree-like shape.
\end{remark}

We will phrase the next two lemmata in the context of quantum graphs as we will need them later on. 

\begin{lemma}
 \label{tree_from_leaves}
Let $G$ be a quantum tree with no vertices of degree $2$. Then the table of distances of all leaves 
determines both the combinatorial tree underlying $G$ and all individual edge lengths.
\end{lemma}
\begin{proof}
Given three leaves $B_i$, $B_j$ and $B_k$ the restriction of the tree to the paths between these leaves is 
shaped like a star. We will denote the length of the three branches by $l_i, l_j$ and $l_k$. The distances between the leaves determine 
the quantities $l_i+l_j$, $l_i+l_k$ and $l_j+l_k$ and thus the three individual lengths $l_i, l_j$ and $l_k$. This means that 
given a path between two leaves $B_i$ and $B_j$ and a third leaf $B_k$ we can find both the point on the path from $B_i$ to $B_j$ 
where the paths from $B_i$ and $B_j$ to $B_k$ branch away and the length of the path from this point to $B_k$. 

We will use this fact repeatedly and proceed by induction on the number of leaves. 

If there are only two leaves the tree consists of a single interval with length the distance between the two leaves. 

Suppose we already have a quantum tree 
with leaves $B_1, \ldots, B_{n-1}$. We now want to attach a new leaf $B_n$. We will first look at the leaves $B_1$ and $B_2$ and
 find the point on the path from $B_1$ to $B_2$ where the paths to $B_n$ branch away. If this point is not a 
vertex of the tree, we create a new vertex and attach the leaf $B_n$ on an edge of suitable length $l_n$. If this point is a vertex of the 
tree we know that the attachment point of $B_n$ has to lie on the subtree branching away from the path from $B_1$ to $B_2$ starting at 
that vertex. Pick a leaf on this subtree, without loss of generality $B_3$, 
and look at the path from $B_1$ to $B_3$. We can again find the point on that path where the paths to $B_n$ branch away. If this point 
is not a vertex of the tree we found the attachment point, otherwise we have reduced our search to a strictly 
smaller subtree. We will now repeat this process. As we reduce the search to a strictly smaller subtree in each step the process 
has to stop after finitely many steps. We will either end up with an attachment point on an edge or on a subtree that consists 
of a single vertex. In either case we can attach the new leaf $B_n$ on an edge of suitable length $l_n$.
\end{proof}

\begin{lemma}
\label{determine_edge_length}
 Let $G_0$ be a $3$-connected combinatorial graph and let $G$ be a quantum graph with underlying combinatorial graph $G_0$. 
Then knowing $G_0$ and the length of each cycle determines the length of each edge in $G$.
\end{lemma}
\begin{proof}
Given an edge $e$ there are at least 3 disjoint paths that connect its end vertices as $G_0$ is $3$-connected. Thus there are two cycles 
in $G_0$ that share the edge $e$ and its end vertices but otherwise are disjoint. Denote these two cycles 
by $c_1$ and $c_2$. Denote the closed walk $c_1 \setminus \{e\} \cup (-c_2 \setminus \{e\})$ by $c_3$. Since $c_1$ and $c_2$ are 
disjoint away from $e$ the closed walk $c_3$ is a cycle. The length of $e$ is given by $2L(e)=L(c_1)+L(c_2)-L(c_3)$ and thus 
determined by the lengths of the cycles.
\end{proof}

\section{Planarity of graphs}

A combinatorial graph is called planar if it admits an embedding into $\R^2$ without edge-crossings. Similarly, a quantum 
graph is planar if the underlying combinatorial graph is planar. In other words, planarity is independent of the edge lengths 
we assigned or of the existence of an isometric embedding. At first glance, planarity seems to be unrelated to the spectrum, and 
indeed, it is not determined by the spectrum of the Laplacian. However, once we consider the entire Bloch spectrum we will 
show that one can determine whether a quantum graph is planar or not.

Given an embedding into $\R^2$ of a planar graph the faces of the embedding are the connected components of $\R^2 \setminus G$.
The edge space of a graph is the $\F_2$-vector space of functions $f: E \ra \F_2$.
The cycle space $\mathcal{C}(G)$ is the subspace generated by all functions that are indicator functions of a cycle in the graph.

\pagebreak

\begin{theorem}
MacLane (1937) \cite{Diestel05}, p.101\\
A graph is planar if and only if its cycle space has a sparse basis.
Sparse means that each edge is part of at most 2 cycles in the basis.
\end{theorem}

\begin{corollary} 
\label{simple=planar}
A graph is planar if and only if it admits a basis of its homology consisting of oriented cycles having no edges of positive overlap.
\end{corollary}
\begin{proof}
Each cycle is confined to a single block of the graph and two cycles in different blocks share at most a single vertex and thus have zero 
overlap. Thus it is sufficient to prove the statement for $2$-connected graphs.

Assume $G$ is planar and $2$-connected and choose an embedding into $\R^2$. The set of boundaries of faces with the exception 
of the outer face forms a 
basis of $\mathcal{C}(G)$ that consists of cycles and is sparse, see \cite{Diestel05}, p.89. We orient all basis cycles counterclockwise 
and get a basis of $H_1(G)$. 
Then no two oriented cycles can run through the same edge in the same direction as no basis cycle can lie inside another basis cycle. 
Thus there are no edges of positive overlap.

Let $\mathcal{B}$ be a basis for $H_1(G)$. If all elements of $\mathcal{B}$ can be represented by oriented cycles, then $\mathcal{B}$ is 
also a basis for the cycle space. 
If $G$ is not planar $\mathcal{B}$ is not sparse by MacLane's theorem. Hence there is an edge in $G$ that is part of three cycles 
in $\mathcal{B}$. No matter how these three cycles are oriented, at least two of them have to go 
through this edge with the same orientation and thus have edges of positive overlap.

Any basis of the homology where every basis element can be represented by a cycle in the graph gives rise to a basis of the cycle 
space consisting of exactly these cycles.
Thus if the graph is not planar any basis of cycles of the homology is not sparse by MacLane's theorem. 
Therefore there exists an edge that is part of three basis cycles. No matter how we orient these three cycles, two of them have to go 
through this edge with the same orientation and thus have edges of positive overlap.
\end{proof}

\begin{definition}
\label{non-positive_basis}
 We call a basis of $H_1(G)$ without edges of positive overlap a {\bf non-positive basis} of the graph and 
remark that a non-positive basis is always sparse.
\end{definition}

If $G$ is $2$-connected and planar we can find a sparse basis by picking the boundaries of faces. This proposition states that the converse is 
true, too.

\begin{proposition}
\cite{MoharThomassen01}
\label{positive=facial}
 Given a sparse basis of the cycle space of a $2$-connected planar graph there exists an embedding into $\R^2$ such that all 
basis elements are boundaries of faces.
\end{proposition}

\section{Dual graphs}

If a graph is planar one can introduce a notion of its dual graph. It is based on the faces of an embedding, that is on the 
sparse basis we defined above. After we have shown that the Bloch spectrum determines planarity we will analyze the 
sparse basis we found further and use it to construct a dual of the graph. This will eventually lead to our theorem that 
$3$-connected planar quantum graphs are completely determined by their spectrum.

We will present two different ways of defining the dual and list some properties.

\begin{definition}
Given a planar graph $G$ we associate to each embedding into the plane a {\bf geometric dual} graph $G^*$. 
The vertices of $G^*$ are the 
faces in the embedding of $G$. The number of edges joining $2$ vertices in $G^*$ is the number of edges 
that the corresponding faces in $G$ have in common. 
\end{definition}

\begin{definition}
 A {\bf cut} of a graph $G$ is a subset of (open) edges $S$ such that $G\setminus S$ is disconnected. A cut is {\bf minimal} if 
no proper subset of $S$ is a cut.
\end{definition}

\begin{definition}
 Given a planar graph $G$, a graph $G^*$ is an {\bf abstract dual} of $G$ if there is a bijective map $\psi : E(G) \ra E(G^*)$ such that 
for any $S \subseteq E(G)$ the set $S$ is a cycle in $G$ if and only if $\psi(S)$ is a minimal cut in $G^*$.
\end{definition}

\begin{proposition}
(\cite{Diestel05}, p.105)
 Any geometric dual of a 2-connected planar graph is an abstract dual and vice versa. 
A planar graph can have multiple non-isomorphic duals. 
Any dual of a planar graph is planar, and $G$ is a dual of $G^*$. If $G$ is $3$-connected, then $G^*$ is unique up to isomorphism.
\end{proposition}

\begin{definition}
 We call two graphs $G$ and $H$ {\bf $2$-isomorphic} if there is a bijection between their edge sets that carries cycles to cycles. Note 
that this does not imply that the graphs are isomorphic.
\end{definition}

\begin{example}
Figure \ref{non_unique_dual} shows two graphs that are $2$-isomorphic but not isomorphic. In one of them the two vertices of degree 
4 are adjacent, in the other one they are not. The third graph is a common dual of them.

\begin{figure}[ht]
\centering
\scalebox{0.85}{\includegraphics{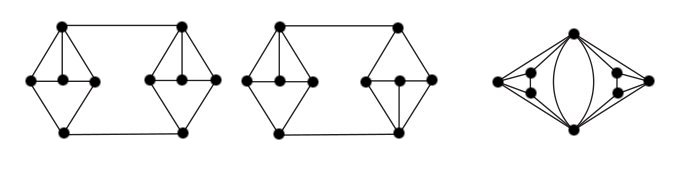}}
\caption{Two $2$-isomorphic graphs and one of their duals}
\label{non_unique_dual}
\end{figure}

\end{example}

\begin{lemma}
\label{2-isomorphic}
Two planar graphs $G$ and $H$ are $2$-isomorphic if and only if they have the same set of abstract duals.
\end{lemma}
\begin{proof}
 Let $\varphi : E(G) \ra E(H)$ be a $2$-isomorphism and let $G^*$ be an abstract dual of $G$ with edge bijection $\psi$. Then 
$\psi \circ \varphi^{-1}$ is an edge bijection that makes $G^*$ an abstract dual of $H$.

Let $G$ and $H$ have the same abstract duals and let $G^*$ be an abstract dual. Let $\psi_1$ be an edge bijection between $G$ and $G^*$ and 
let $\psi_2$ be an edge bijection between $H$ and $G^*$. Then $\psi_2^{-1}\circ\psi_1$ is a $2$-isomorphism between $G$ and $H$.
\end{proof}

\section{Spectra of combinatorial graphs}

The material in this section is mostly taken from Chung's book \cite{Chung97}. The spectral theory of combinatorial graphs is 
very different from that of quantum graphs because the spectrum is finite in this case. Nevertheless, if a quantum graph is 
equilateral there is a close relation between its spectrum and the spectrum of the underlying combinatorial graph. We will 
use this relation to carry over some examples of graph-isospectrality to the quantum graph setting.

\begin{definition}
\label{weight_function}
We define a {\bf weight function} $w : V \times V \ra \N$ 
on a graph as follows. If $v \neq v'$ then $w(v,v')$ is the number of edges between $v$ and $v'$. On the diagonal $w(v,v)$ is 
half the number of loops attached at $v$. Note that the degree of a vertex is given by $deg(v)=\sum_{v'\in V}w(v,v')$.  
\end{definition}

\begin{definition}
Let $f: V \ra \R$ be a function on the vertices of a combinatorial graph. Then the {\bf combinatorial Laplacian} $\Delta^C$ acts as follows:
\begin{equation*}
 \Delta^C f(v) := f(v)- \sum_{v' \in V} \frac{w(v,v')f(v')}{\sqrt{deg(v)deg(v')}}
\end{equation*}
\end{definition}

\begin{lemma}
 The operator $\Delta^C$ can be written as the matrix
\begin{equation*}
 (\Delta^C)_{v,v'} = \begin{cases} 1-\frac{w(v,v)}{deg(v)} & v=v', \\
		 -\frac{w(v,v')}{\sqrt{deg(v)deg(u)}} & v \neq v'
\end{cases}
\end{equation*}
where we see functions on the combinatorial graph as vectors in $\R^V$.
The spectrum of the operator $\Delta^C$ is the (finite) list of eigenvalues of this matrix including multiplicities.
\end{lemma}

\begin{definition}
The {\bf combinatorial spectrum} of a combinatorial graphs is the spectrum of $\Delta^C$. 
 If two combinatorial graphs have the same combinatorial spectrum we say they are {\bf graph-isospectral}.
\end{definition}

\begin{remark}
 There are alternative definitions of the spectrum of a combinatorial graph. Some authors (for example \cite{CDS95}) 
call the collection of eigenvalues of the adjacency matrix the spectrum of the graph. However, all these definitions give 
rise to the same notion of isospectrality.
\end{remark}

\begin{proposition}
\cite{Chung97}
 The combinatorial spectrum uniquely identifies complete combinatorial graphs.
\end{proposition}

\begin{proposition}
\cite{Chung97}
 The combinatorial spectrum of a graph determines whether or not it is bipartite.
\end{proposition}

\begin{theorem}
\label{isospec_combo_graphs}
\cite{BrGoGu98}, \cite{Seress00}
There are families of non-isomorphic graph-isospectral combinatorial graphs of size that grows exponentially in the number of edges.
\end{theorem}
Brooks, Gornet and Gustafson  construct families of regular graph-isospectral combinatorial graphs on $n$ edges of size 
$\frac{1}{2}e^{\ln(2)n/12}$
in \cite{BrGoGu98}. These graphs contain loops and multiple edges. Seress builds graph-isospectral families of simple regular graphs of 
size $\frac{1}{80}e^{\ln(2)n/24}$ in \cite{Seress00}.

\begin{remark}
 Both of the above constructions are based on the method of Seidel switching. 
The base case works as follows. Let $G_1=(V_1, E_1)$ and $G_2=(V_2, E_2)$ be two regular simple combinatorial graphs. We will 
now construct two new combinatorial graphs $G$ and $\tilde{G}$ that are graph-isospectral. The vertex set of $G$ and 
$\tilde{G}$ is $V_1 \sqcup V_2$. All edges in $G_1$ and $G_2$ are also edges in $G$ and $\tilde{G}$. The set of edges in $G$ 
between $V_1$ and $V_2$ satisfies the following rule. Each vertex in $V_1$ is adjacent to exactly half the vertices in $V_2$ 
and every vertex in $V_2$ is adjacent to exactly half the vertices in $V_1$. The edges in $\tilde{G}$ between $V_1$ and $V_2$ 
are obtained through `switching edges on and off'. Whenever there is an edge between $v_1\in V_1$ and $v_2 \in V_2$ in $G$ 
there is no edge between 
these two vertices in $\tilde{G}$. Whenever there is no edge between $v_1\in V_1$ and $v_2 \in V_2$ in $G$ there is an edge 
between these two vertices in $\tilde{G}$. That is, we switch all the edges between $V_1$ and $V_2$ to non-edges and vice versa.

In order to obtain large graph-isospectral families one has to generalize this method. One uses $k$ regular simple 
combinatorial graphs and switches edges on and off between them, see \cite{BrGoGu98} for a rigorous statement of the 
general theorem.
\end{remark}

\chapter{Quantum graphs and differential forms}

\begin{definition}
 A {\bf quantum graph} $G$ consists of the following data:
\begin{enumerate}
\item A finite combinatorial graph with edge set $E$ and vertex set $V$
\item A length function $L: E \ra \R_{>0}$ that assigns a length to each edge
\end{enumerate}
\end{definition}

Let $\{ e \sim v \}$  denote the set of edges $e$ adjacent to a vertex $v$.

Let $\mathcal{L}:=\sum\limits_{e\in E}L(e)$ denote the total edge length of the quantum graph.

\begin{remark}
 A quantum graph has a natural topology and structure as a $1$-dimensional CW complex. This gives a way to define the homology and 
cohomology of the quantum graph. 
\end{remark}

The concept of differential forms on a quantum graph was introduced in \cite{GaveauOkada91}. 

\begin{definition}
 A {\bf vector field} $X$ on $G$ consists of a vector field on each edge. We see each edge as a closed interval, that is, as a $1$-dimensional 
manifold, and use the associated notion of vector field. In particular a vector field is multivalued 
at the vertices. 
\end{definition}

Let $\nu_{v,e}$ denote the outward unit normal for the edge $e$ at the vertex $v$, where again we see the edge $e$ as a 
$1$-dimensional manifold with boundary.
Let $X_1$ be an auxiliary vector field that is real and has constant length $1$ on all edges.

\begin{definition}
A {\bf $0$-form} $f$ on $G$ is a function that is $C^{\infty}$ on the edges, that is continuous, and that satisfies the Kirchhoff boundary 
condition 
\begin{equation*}
\sum_{e \sim v }\nu_{v,e}(f|_e)=0 
\end{equation*}
at all vertices $v \in V$. We denote the space of $0$-forms by $\Lambda^0$.
\end{definition}

\begin{definition}
A {\bf $1$-form} $\alpha$ on $G$ consists of a smooth $1$-form $\alpha_e$ on each closed edge $e$ such that $\alpha$ satisfies 
the boundary condition
\begin{equation*}
 \sum_{e \sim v} \alpha_e(\nu_{v,e})=0
\end{equation*}
at all vertices $v \in V$. We denote the space of $1$-forms by $\Lambda^1$.
\end{definition}

\begin{definition}
 For a real $1$-form $\alpha$ we define the operator $d_{\alpha} : \Lambda^0 \ra \Lambda^1$ through the requirement 
\begin{equation*}
(d_{\alpha}f)(X):=X(f)+2\pi i\alpha(X)f
\end{equation*}
for all vector fields $X$. We denote the operator $d_0$ by $d$.
\end{definition}

\begin{remark}
Note that the boundary conditions for functions and $1$-forms are compatible. A function $f$ satisfies Kirchhoff 
boundary conditions if and only if $df$ satisfies the boundary condition for $1$-forms. 
\end{remark}

\begin{definition}
 We define a hermitian inner product on $\Lambda^0$ by 
\begin{equation*}
 (f,g) :=\int_G f(x)\overline{g(x)}dx
\end{equation*}
\end{definition}

\begin{definition}
\label{hermitian_inner_product}
We define a hermitian inner product on $\Lambda^1$ by 
\begin{equation*}
 (\alpha,\beta) :=\int_{G}\alpha(X_1)\overline{\beta(X_1)}dx=\sum_{e\in E}\int_0^{L(e)}\alpha_e(X_1|_e)\overline{\beta_e(X_1|_e)}dx
\end{equation*}
This is clearly independent of the choice of the auxiliary vector field $X_1$.
\end{definition}

We are now going to define the formal adjoint of $d_{\alpha}$. Formally it should satisfy
\begin{equation*}
 (d_{\alpha}^* \beta, f)=(\beta, d_{\alpha}f)
\end{equation*}
for all $f \in \Lambda^0$. We have\
\begin{eqnarray*}
 (\beta, d_{\alpha}f) &=& \int_G\beta(X_1)\overline{X_1(f)}dx - 2\pi i\int_G\beta(X_1)\overline{\alpha(X_1)f}dx \\
&=& -\int_G X_1(\beta(X_1))\overline{f}dx+\sum_{v\in V } \overline{f(v)}\sum_{e \sim v} \beta_e(\nu_{v,e})\\
&& -2\pi i\int_G\alpha(X_1)\beta(X_1)\overline{f}dx 
\end{eqnarray*}
where we used integration by parts. The sum term vanishes because of the boundary condition on $1$-forms.
So we find that $d^*_{\alpha}$ satisfies
\begin{equation*}
 d_{\alpha}^* \beta = - X_1(\beta(X_1))-2\pi i\alpha(X_1) \beta(X_1) = d^*\beta - 2\pi i \alpha(X_1)\beta(X_1)
\end{equation*}
which again is independent of the choice of $X_1$.

\begin{definition}
 For each edge $e \in E$ we define the Sobolov space $W_2(e)$ as the closure of $C^2([0,L(e)])$ with respect to the norm 
$||f||_2^2:=\sum\limits_{j=0}^2 \int\limits_0^{L(e)} |f^{(j)}(x)|^2dx$.
 
We define the global Sobolov space $W_2(G)$ as the space of all functions $f$ that are continuous on the entire graph and that 
satisfy $f|_e \in W_2(e)$ for all $e\in E$.
\end{definition}

\begin{definition}
 We define a {\bf Schr\"odinger type operator} 
\begin{equation*}
 \Delta_{\alpha}:=d_{\alpha}^* d_{\alpha}
\end{equation*}
 on $\Lambda^0$. We extend its domain to
\begin{eqnarray*}
 Dom(\Delta_{\alpha}):=\left\{ f \in W_2(G) \middle| \forall v\in V : \sum_{e \sim v}\nu_{v,e}(f|_e)=0  \right\}
\end{eqnarray*}
\end{definition}

\begin{remark}
 Note that the introduction or removal of vertices of degree $2$ would not change the space $Dom(\Delta_{\alpha})$. This justifies 
our assumption that all graphs do not have vertices of degree $2$.
\end{remark}

\begin{definition}
 We denote the the set of eigenvalues, ie the {\bf spectrum of $\Delta_{\alpha}$} including multiplicities by $Spec_{\alpha}(G)$.
\end{definition}

\begin{proposition}
\cite{Kuchment04}
 The operator $\Delta_{\alpha}$ is elliptic. The spectrum is discrete, infinite, bounded from below, with a single accumulation point 
at infinity. The multiplicity of each eigenvalue is finite.
\end{proposition}

\begin{theorem}
 \cite{GaveauOkada91}
We have $H^1(G,\C)=\Lambda^1 \slash d(\Lambda^0)$. 
Thus the definitions of $1$-forms and $0$-forms produce the expected deRham cohomology.
\end{theorem}

\chapter{A trace formula}

In this chapter we will present the trace formula we are going to work with. Although there are many different versions of 
it (see \cite{BolteEndres08} for a survey) all of them have essentially the same structure. On the left side there 
is an infinite sum of some test function evaluated at all the eigenvalues including multiplicity. The right side contains 
a term involving the total edge length of the quantum graph, an index term that simplifies to the Euler characteristic for Kirchhoff 
boundary conditions, and an infinite sum over the periodic orbits in the quantum graph.

We will use the following.

\begin{theorem}
\cite{KottosSmilansky99}
\label{trace_formula}
 The spectrum $Spec_{\alpha}(G)=\{k_n^2\}_n$ of the operator $\Delta_{\alpha}$ determines the following exact wave trace formula.
\begin{eqnarray*}
\sum_n \delta(k-k_n)
=\frac{\mathcal{L}}{\pi}+\chi(G)\delta(k)+
\frac{1}{2\pi} \sum_{p\in PO}\left(\mathcal{A}_p(\alpha)e^{ik l_p}+\overline{\mathcal{A}_p(\alpha)}e^{-ik l_p}\right)
\end{eqnarray*}
Here the first sum is over the eigenvalues including multiplicities, the $\delta$ are Dirac $\delta$ distributions.

$\mathcal{L}$ denotes the total edge length of the quantum graph. $\chi(G)=V-E-1$ denotes the Euler characteristic.

The second sum is over all periodic orbits, $l_p$ denotes the length of a periodic orbit. 
A periodic orbit is an oriented closed walk in the quantum graph (without a fixed starting point).

The coefficients $\mathcal{A}_p(\alpha)$ are given by
\begin{equation*}
 \mathcal{A}_p(\alpha)=\tilde{l}_pe^{2\pi i\int_p\alpha}\prod_{b\in p} \sigma_{t(b)}
\end{equation*}
Here $\tilde{l}_p$ is the length of the primitive periodic orbit that $p$ is a repetition of. The $e^{2\pi i\int_p\alpha}$ is 
the phase factor or `magnetic flux'. The product is over the sequence of oriented edges or bonds in the periodic orbit. 
The coefficient $\sigma_{t(b)}$ at the terminal vertex $t(b)$ of each bond $b$ is called the vertex scattering  coefficient 
and is given by $\sigma_{t(b)}=-\delta_{t(b)}+\frac{2}{\deg(t(b))}$. Here $\delta_{t(b)}$ is defined to be equal 
to one if the periodic orbit is backtracking at the vertex $t(b)$ and zero otherwise. 
\end{theorem}

\begin{proof}
 We will not give a complete proof of the trace formula here but merely give a sketch of the proof and provide 
some of the key ideas that go into proving the trace formula. Our presentation here summarizes the proof given in 
\cite{KottosSmilansky99}.

We will fix an orientation and a parametrization for the edges, that is we consider a directed quantum graph. Then every 
bond $b=(v,v')$ has a natural orientation from $v$ to $v'$ and a reversed orientation from $v'$ to $v$ denoted by $\overline{b}$.
We identify the bond $b$ with the interval $[0,L(b)]$. Let $X_b$ be the vector field of unit length along the bond $b$ that 
points in the direction of the terminal vertex of $b$.

We first observe that all eigenfunctions of the Schr\"odinger operator $\Delta_{\alpha}$ are sine waves on the individual edges. 
The eigenfunctions are always of the form
\begin{align*}
 \psi_b(x)= e^{-2\pi i \int_0^{x}\alpha_b}\left( a_b e^{ikx} + \tilde{a}_b e^{-ikx}\right)
\end{align*}
for some parameters $a_b$ and $\tilde{a}_b$. Note that we always have $\psi_b(x)= \psi_{\overline{b}}(L(b)-x)$. 
The value $k^2$ is an eigenvalue of the quantum graph whenever there exists a choice of the values of the parameters 
$a_b$ and $\tilde{a}_b$ on all the bonds such that the Kirchhoff boundary conditions are satisfied at all the vertices. 
This gives rise to a system of $2E$ linear equations that can be written in finite dimensional matrix form. This is the key 
step where quantum graphs behave better than arbitrary manifolds, this reduction to a finite problem is ultimately the reason 
why we have an exact trace formula for quantum graphs instead of just an asymptotic approximation.
We let
\begin{align*}
 S(k,\alpha):=D(k,\alpha)T
\end{align*}
 Here $D(k,\alpha)_{bb'}:=\delta_{bb'}e^{ikL(b)-2\pi i\int_0^{L(b)} \alpha_b}$ is a diagonal matrix that encodes the metric 
structure and the $1$-form $\alpha$. The matrix $T$ encodes the combinatorial structure of the graph. If the terminal vertex of 
the bond $b$ is the initial vertex of the bond $b'$ we set $T_{bb'}:=\sigma_{t(b)}=-\delta_{t(b)}+\frac{2}{\deg(t(b))}$ to be the 
vertex scattering coefficient at the vertex $t(b)$, here $\delta_{t(b)}$ is equal to $1$ if $b =\overline{b'}$ and 
$0$ otherwise; we set $T_{bb'}:=0$ otherwise.
The eigenvalues are now given by the secular equation
\begin{align*}
 \zeta(k):=\det (Id_{2n} - S(k,\alpha)) =0
\end{align*}
Note that this equation is also used for numeric computations of the eigenvalues.
We can now apply a counting operation to $\zeta(k)$, we define
\begin{align*}
d(k)&:=-\frac{1}{\pi}\lim_{\varepsilon \ra 0} Im \frac{\partial}{\partial k}\log \zeta(k+i\varepsilon)\\
 N(k')&:=\int_{k_0-\varepsilon'}^{k'}d(k)dk
\end{align*}
One can show using the Taylor expansion of $\zeta$ that $N(k')$ counts the zeros of $\zeta(k)$ (including multiplicities) in 
the interval $[k_0,k')$. In other words, as a distribution we have 
\begin{align*}
 d(k)=\sum_n \delta(k-k_n)
\end{align*}
This means we have found the following distributional equality.
\begin{align*}
 \sum_n \delta(k-k_n) = 
-\frac{1}{\pi}\lim_{\varepsilon \ra 0} Im \frac{\partial}{\partial k}\log \det \left(Id_{2n} - S(k+i\varepsilon,\alpha)\right)
\end{align*}
 It equates a sum over the eigenvalues with a quantity that solely depends 
on the combinatorial and metric features of the quantum graph. To get the expression of the trace formula as stated above 
one has to perform a series of sophisticated and clever algebraic manipulations.
\end{proof}

\begin{remark}
\label{contractible}
The phase factor $e^{2\pi i\int_p\alpha}$ of a periodic orbit only depends on its homology class by Stokes theorem. For a 
contractible periodic orbit it is equal to $1$. 
\end{remark}

\begin{corollary}
\label{fourier_trace_formula}
\cite{KottosSmilansky99}
The Fourier transform of this trace formula is given by:
\begin{equation*}
\sum_n e^{-il\lambda_n}=2\mathcal{L}\delta(l)+\chi(G)+ 
\sum_{p\in PO}\mathcal{A}_p(\alpha)\delta(l-l_p)+\overline{\mathcal{A}_p(\alpha)}\delta(l+l_p)
\end{equation*} 
\end{corollary}

\chapter{Finiteness of quantum isospectrality}

The goal of this chapter is to prove that any family of quantum graphs whose Laplace operators are isospectral 
is finite, and to provide an upper bound for the size of isospectral families.

We will analyze the trace formula for $\Delta_0$ from Theorem \ref{trace_formula}.
First we need a lemma that will guarantee the non-vanishing of certain terms in the trace formula. 

\begin{lemma}
\label{positivity}
Assume $G$ is a leafless quantum graph. Then we have
\begin{enumerate}
 \item If $p$ is a periodic orbit with an even number of backtracks then $\mathcal{A}_p > 0$.
\item If $p$ is a periodic orbit with an odd number of backtracks then $\mathcal{A}_p < 0$.
\end{enumerate}
\end{lemma}
\begin{proof}
The coefficients $\mathcal{A}_p$ for the standard Laplacian are given by
\begin{equation*}
 \mathcal{A}_p=\tilde{l}_p\prod_{b\in p} \sigma_{t(b)}
\end{equation*}
The `magnetic flux' term is just $1$ in this case.

The length of a periodic orbit is always strictly positive and thus has no influence on the sign of $\mathcal{A}_p$. The vertex 
scattering coefficient is $\frac{2}{\deg{v}}$ if the periodic orbit is not back scattering and $-1+\frac{2}{\deg{v}}$ if it is 
back scattering. As we assumed $deg(v) \ge 3$ the sign of the product is equal to the parity of the number of back tracks.
\end{proof}

\begin{remark}
This means that the coefficient for a single periodic orbit is always nonzero. In general however in can happen that several different 
periodic orbits have 
exactly the same length and the sum of their coefficients is zero so one would not see them in the trace formula. 
It is shown in \cite{GutkinSmilansky01} that quantum graphs with rationally independent edge lengths are spectrally 
determined. It is one of the key steps in 
their proof to identify the set of edge lengths through the periodic orbits that just backtrack twice on a single edge. The next 
example shows that this step fails if the rational independence hypothesis is dropped.
\end{remark}

\begin{example}

The quantum graph in Figure \ref{qgraph_vanish_trace} has an edge of length $\frac{23}{10}$. In order to see this from the 
trace formula one would 
look for periodic orbits of length $\frac{23}{5}$. However there are multiple periodic orbits of length $\frac{23}{5}$ and the sum 
of all their coefficients vanishes.

\begin{figure}[ht]
\centering
\scalebox{1.1}{\includegraphics{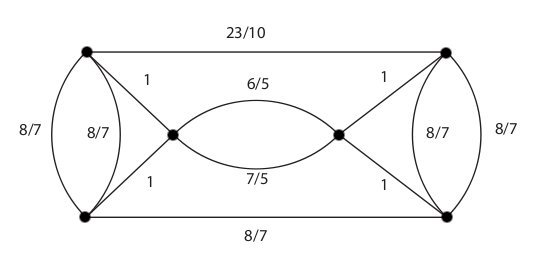}}
\caption{A quantum graph with vanishing terms in the trace formula}
\label{qgraph_vanish_trace}
\end{figure}

A periodic orbit of length $\frac{23}{5}$ could either contain two edges of length $\frac{23}{10}$ or two edges of length $1$, one edge 
of length $\frac{6}{5}$ and one edge of length $\frac{7}{5}$ or finally one edge of length $1$ and three edges of length 
$\frac{6}{5}$. There is one periodic orbit containing two edges of length $\frac{23}{10}$, its coefficient is 
$\mathcal{A}_{p}=\frac{23}{5} \cdot (-\frac{1}{2})(-\frac{1}{2})=\frac{23}{20}$. There are four periodic orbits of the second type, 
all of them have coefficient $\mathcal{A}_p=\frac{23}{5} \cdot (\frac{1}{2})^3(-\frac{1}{2})=-\frac{23}{80}$. There are no 
periodic orbits of the third type. Thus the sum of all coefficients of periodic orbits of length $\frac{23}{5}$ is zero. Thus 
one would not notice the existence of an edge of length $\frac{23}{10}$ simply by looking for periodic orbits of length $\frac{23}{5}$. 
As there are other periodic orbits in the quantum graph that contain this edge one might infer its existence from these.

Although the given example is not a simple graph it is possible to construct examples of this phenomena with simple graphs.
\end{example}

\begin{remark}
\label{loop_half_count}
Loops present a minor technical difficulty for some of the arguments in this chapter. Whenever a quantum graph has a loop of length 
$l$ we will count it as two edges of length $l/2$. This comes from the fact that we make statements about the shortest edge length 
but what we show are statements about the shortest periodic orbit in the quantum graph. If a graph does not have loops the shortest 
periodic orbit has length equal to twice the shortest edge length. If an edge is a loop the corresponding periodic orbit 
only has the length equal to once the length of this loop.

This mirrors a similar special treatment of loops in combinatorial graphs, see 
Definition \ref{weight_function}. Using this convention we can still say that any periodic orbit contains at least two edges.
\end{remark}

\begin{lemma}
\label{same_minimal_length}
All leafless quantum graphs isospectral to a given quantum graph have the same shortest edge length.
The multiplicity of that edge length may vary.
\end{lemma}
\begin{proof}
This follows from Lemma \ref{positivity}. The shortest periodic orbit in a quantum graph consists of twice the shortest edge length. If 
there are several edges of the same shortest edge length there might be several periodic orbits of minimum length but all of them 
have either two or zero backtracks. The zero backtrack case 
only happens if there is a double edge where both edges have the shortest edge length. Thus all the $\mathcal{A}_p$ coefficients of 
these periodic orbits will be positive and their sum cannot vanish.
\end{proof}

\begin{remark}
If we allow quantum graphs with leaves, the coefficients $\mathcal{A}_p$ for the shortest periodic orbits can be positive or 
negative. It is 
not hard to construct an example where the sum of these coefficients cancels out to zero. This means there might be isospectral 
quantum graphs with leaves with different minimum edge lengths.
\end{remark}

\begin{lemma}
\label{half_linear_combination}
Let $G$ and $G'$ be two leafless quantum graphs that are isospectral. Then all edge lengths of $G'$ 
are $\frac{1}{2}\N$-linear combinations of edge lengths occurring in $G$. 
\end{lemma}
\begin{proof}
We will show this by contradiction.
Clearly all periodic orbits in $G$ have lengths that are $\N$-linear combinations of the edge lengths occurring in $G$. Thus all 
terms in the trace formula for $G$ will occur at lengths that are $\N$-linear combinations of the edge lengths occurring in $G$.

Suppose  $G'$ has edge lengths that are not $\frac{1}{2}\N$-linear combinations of edge lengths occurring in $G$.
Consider the shortest length $L$ of periodic orbits that involve at least one edge length that is not a $\frac{1}{2}\N$-linear 
combination of edge lengths occurring in $G$. Any periodic orbit of length $L$ contains exactly one edge length that
is not a $\frac{1}{2}\N$-linear combination of edge lengths occurring in $G$. If it would involve two distinct new lengths 
the periodic orbit consisting of twice the shorter of the two new length would be shorter. There are two possibilities of 
what the periodic orbits of length $L$ can look like. Either they consist of twice the same new edge length with two or zero 
backtracks, or they form a closed walk in the graph that contains a new edge length only once. If the periodic orbit 
contains a new edge length only once it has to be homologous to a cycle in the graph, as it has minimum length it is 
this cycle and thus contains no back tracks. Note that both cases can occur at the same time involving different new 
edge lengths.
Thus we have shown that any periodic orbit of length $L$ has an even number 
of back tracks and thus a positive coefficient $\mathcal{A}_p$ by Lemma \ref{positivity} and the trace formula of $G'$ will 
have a non-vanishing coefficient at length $L$.
The length $L$ is not an $\N$-linear combinations of the edge lengths occurring in $G$ in either of our two cases, thus the 
trace formula for $G$ does not involve a term at length $L$. Therefore the trace formulas for $G$ and $G'$ do not match 
and the quantum graphs are not isospectral.
\end{proof}

\begin{remark}
 There are quantum graphs with edges of integer and half integer length such that all periodic orbits have 
integer length. Figure \ref{integer_periodic_orbits} shows an example of such a quantum graph.
 
\begin{figure}[ht]
\centering
\scalebox{0.9}{\includegraphics{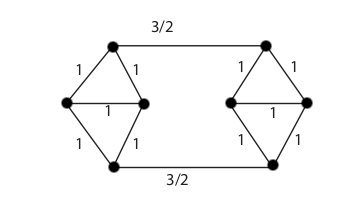}}
\caption{A quantum graph with half integer edge length and only integer length periodic orbits}
\label{integer_periodic_orbits}
\end{figure}
\end{remark}

\begin{remark}
This lemma fails for graphs with degree 1 vertices or for a more general operator of the form $d_{\alpha}^*d_{\alpha}$. 
For the standard Laplacian on arbitrary quantum graphs one can show a similar lemma saying that all edge lengths are 
$\frac{1}{2^k}\N$-linear combinations of 
the edge lengths occurring in $G$ for some $k$ but it is not clear whether there is a bound on $k$.
For a general operator, this kind of argument does not imply any restrictions because we have no information about the coefficients 
of periodic orbits that are cycles in the graph.
\end{remark}

\begin{theorem}
\label{finite_isospectrality}
All families of leafless quantum graphs that are isospectral for the standard Laplacian are finite.
\end{theorem}
\begin{proof}
We will prove the following equivalent statement:

\itshape
For a given leafless quantum graph $G$ there are at most finitely many non-isomorphic leafless quantum graphs 
that are isospectral to $G$ for the standard Laplacian.
\upshape

Without loss of generality we will assume that the shortest edge length in $G$ is $1$.
 
Let $\mathcal{L}$ be the total edge length of $G$. Let $G'$ be a quantum graph that is isospectral 
to $G$ for the standard Laplacian. Then $G'$ will also have total edge length $\mathcal{L}$ by the trace formula 
\ref{trace_formula}, and minimum edge length $1$ by Lemma \ref{same_minimal_length}. All edge lengths occurring in $G'$ will be 
$\frac{1}{2}\N$-linear combinations of the edge lengths in $G$ by Lemma \ref{half_linear_combination}. Thus there exists 
a finite list of possible edge lengths that can occur in $G'$. 

The quantum graph $G'$ can have at most $\lfloor \mathcal{L} \rfloor$ edges. There are only finitely many non-isomorphic 
combinatorial graphs with $\lfloor \mathcal{L} \rfloor$ or less edges. Thus the underlying combinatorial graph of $G'$ has to 
be one of this finite list of graphs with at most $\lfloor \mathcal{L} \rfloor$ edges. 

There is only a finite number of ways to assign the finite number of possible edge lengths to each of the finitely many possible 
combinatorial graphs. Thus there are only finitely many quantum graphs that are isospectral to $G$.
\end{proof}

\section{An explicit bound on the size of isospectral families}

We can find a  lower  bound on the size of isospectral families by looking at large families of isospectral combinatorial graphs 
and then apply the following result.

\begin{theorem}
 \cite{Cattaneo97}
\label{isospec_combo_quantum}
Let $G$ be a regular equilateral quantum graph. Then the spectrum of $G$ is completely and explicitly determined by the 
spectrum of the underlying combinatorial graph.

In particular if two regular equilateral quantum graphs have graph-isospectral underlying combinatorial graphs they are isospectral 
as quantum graphs as well.
\end{theorem}

\begin{corollary}
 There are families of isospectral equilateral pairwise non-isomorphic quantum graphs whose size grows exponentially 
in the number of edges. If we normalize the edge length to be $1$ the size of the families grows exponentially in the 
total edge lengths of the quantum graphs.
\end{corollary}
\begin{proof}
This is a direct consequence of the existence of families of graph-isospectral pairwise non-isomorphic combinatorial graphs of 
size growing exponentially in the number of edges, see \cite{BrGoGu98} and \cite{Seress00}, Theorem \ref{isospec_combo_graphs}, 
combined with Theorem \ref{isospec_combo_quantum} above.
\end{proof}

\begin{lemma}
\label{bound_number_of_edges}
 Any leafless quantum graph with total edge length $\mathcal{L}$, Euler characteristic $\chi$ and minimum edge length $1$ can have at most 
\begin{equation*}
 M:=\min \{ \lfloor \mathcal{L} \rfloor, 3\chi-3 \}
\end{equation*}
edges.
\end{lemma}
\begin{proof}
 If the quantum graph has minimum edge length $1$ and total edge length $\mathcal{L}$ it can have at most $\lfloor \mathcal{L} \rfloor$ 
edges.

If a quantum graph is leafless and does not have vertices of degree $2$ all its vertices have degree at least $3$. This implies 
\begin{equation*}
 E \ge \frac{3}{2}V
\end{equation*}
On the other hand we have
\begin{equation*}
 E-V+1 = \chi
\end{equation*}
Put together this implies
\begin{equation*}
 E \le 3\chi -3 
\end{equation*}
\end{proof}

\begin{remark}
 The same equations also yield the two bounds
\begin{equation*}
 \frac{1}{2}V+1 \le \chi \le E
\end{equation*}

\end{remark}

\begin{definition}
 We define a {\bf list of possible edge lengths} to be a list of edge lengths (possibly with repetition) that could form the set of all 
edge lengths occurring in a quantum graph that is isospectral to $G$. 

Every such list satisfies the following properties.
\begin{enumerate}
\item Every edge length in the list is a $\frac{1}{2}\N$-linear combination of the edge lengths in $G$.
\item The sum of the edge lengths of all the edges in the list is $\mathcal{L}$.
\item Every edge length in $G$ is a $\frac{1}{2}\N$-linear combination of the edge lengths in the list.
\item The shortest edge length occurring in the list is $1$.
\end{enumerate}
Note that each such list contains at most $M$ items by Lemma \ref{bound_number_of_edges}.
\end{definition}

\begin{lemma}
\label{possible_edge_length_sets}
There are at most $M^{4\lfloor \mathcal{L} \rfloor}$ different lists of possible edge lengths.
\end{lemma}
\begin{proof}
Let $1 =l_1 \le \hdots \le l_n$ denote the edge lengths in $G$ (including repetitions). Let $1=e_1 \le \hdots \le e_m$ denote a 
list of possible edge lengths. We then have $n \le M$ and $m \le M$ by Lemma \ref{bound_number_of_edges}. 
Every $e_j$ can be written as $e_j = \sum_{i=1}^n\alpha_i^j l_i$ for 
some coefficients $\alpha_i^j \in \{ 0, \frac{1}{2}, 1, \frac{3}{2}, 2, \hdots \}$. We also have
\begin{eqnarray*}
 \mathcal{L} = \sum_{j=1}^m e_j =\sum_{j=1}^m \sum_{i=1}^n\alpha_i^j l_i \ge \sum_{j=1}^m \sum_{i=1}^n\alpha_i^j
\end{eqnarray*}
In order to count the number of lists of possible edge lengths we have to choose $nm \le M^2$ coefficients 
$\alpha_i^j$ such that the sum over all of them is at most 
$\mathcal{L}$. The number of possibilities can be bounded as follows. Start with all $a_i^j$ zero and then choose one coefficient 
$\alpha_i^j$ and increase it by $\frac{1}{2}$, do this $2\lfloor\mathcal{L}\rfloor$ times. This gives an upper bound of 
$\left(M^2\right)^{2\lfloor\mathcal{L}\rfloor}= M^{4\lfloor\mathcal{L}\rfloor}$
for the total number of lists of possible edge length.

Note that this bound does not use the property that the $l_i$ are $\frac{1}{2}\N$-linear combinations of the $e_j$. 
\end{proof}

\begin{theorem}
\label{isospectral_bound}
 Any family of isospectral leafless quantum graphs with common minimum edge length $1$ and total edge length $\mathcal{L}$ is 
at most of size
\begin{equation*}
(\frac{2}{3}M+1)^{2M}M!M^{4\lfloor\mathcal{L}\rfloor} 
\end{equation*}
This means that the size of isospectral families can be bounded by $\mathcal{L}^{7\mathcal{L}}=e^{7\mathcal{L}\log(\mathcal{L})}$.
\end{theorem}
\begin{proof}
The graph $G$ has at most $M$ edges, as it is leafless every vertex has degree at least three so $G$ has at most 
$\frac{2}{3}M$ vertices. 
We need to bound the number of combinatorial graphs with at most $M$ edges on at most 
$\frac{2}{3}M$ vertices.
There are $(\frac{2}{3}M)^2$ possibilities for the end vertices of each edge, so there are at most 
$(\frac{2}{3}M)^{2M}$ 
combinatorial graphs with $M$ edges on $\frac{2}{3}M$ vertices. As this includes graphs 
with isolated vertices this is also a bound for graphs with at most $\frac{2}{3}M$ vertices.
To bound the number of graphs with at most $M$ edges we will add in a `kill vertex' and say that any edge that has the 
`kill vertex' at one of its ends is not part of the graph. This gives the bound $(\frac{2}{3}M+1)^{2M}$ for the number of 
combinatorial graphs with at most $\frac{2}{3}M$ vertices and at most $M$ edges.

By Lemma \ref{possible_edge_length_sets} we have at most $M^{4\lfloor\mathcal{L}\rfloor}$ lists of possible edge length.

If we are given a combinatorial graph with at most $M$ edges and a list of at most $M$ possible edge lengths 
there are at most $M!$ ways to assign the edge lengths to the graph. We are ignoring the fact that the combinatorial 
graph might not have the same number of edges as the list of possible edge lengths in which case there would be zero ways to assign 
the lengths.

Putting the three estimates together the maximal size of an isospectral family is bounded by
\begin{equation*}
 (\frac{2}{3}M+1)^{2M}M!M^{4\lfloor\mathcal{L}\rfloor}
\end{equation*}
To get the bound that involves only the total edge length we note that $M \le \mathcal{L}$ by the definition of $M$ and that $M! < M^M$.
\end{proof}

\chapter{Defining the Albanese torus of a quantum graph}

The Jacobian and its dual, the Albanese torus have been studied for combinatorial graphs, see \cite{Nagnibeda97} 
and \cite{KotaniSunada00}. We will generalize this to 
quantum graphs. If the quantum graph is equilateral, our definition recovers theirs.

The Albanese torus or Albanese variety is a classic invariant of algebraic varieties and complex manifolds studied since the 
19th century. Both the Albanese torus 
and the Jacobian carry a natural complex structure induced from the space of holomorphic $1$-forms. The complex structure 
does not carry over to combinatorial or quantum graphs, in fact the tori don't even have to be even dimensional. On the other hand, 
the inner product structure is induced from the inner product on harmonic $1$-forms and this idea carries over to combinatorial 
and quantum graphs.

\begin{definition}
We call a $1$-form $\alpha$ {\bf harmonic} if $d^*\alpha \in \Lambda^0$ and $dd^*\alpha = 0$. 
\end{definition}

\begin{lemma}
\cite{GaveauOkada91}
A $1$-form $\alpha$ is harmonic if and only if $\alpha(X_1)$ is constant on all edges where $X_1$ is the auxiliary vector 
field of constant length $1$.
\end{lemma}

\begin{lemma}
\cite{GaveauOkada91} 
Any $\beta \in \Lambda^1$ admits a unique Hodge decomposition of the form
\begin{equation*}
 \beta=d\psi+\tilde{\beta}
\end{equation*}
where $\psi \in \Lambda^0$ and $\tilde{\beta}$ is harmonic.

Thus each cohomology class has exactly one harmonic representative.

If $\beta$ is real, then so are $\psi$ and $\tilde{\beta}$.
\end{lemma}

\begin{definition}
\label{inner_prod_cohomology}
 We define an {\bf inner product on $H^1(G,\R)$} by 
\begin{equation*}
 ([\alpha],[\beta]):=(\tilde{\alpha}, \tilde{\beta})
\end{equation*}
where $\tilde{\alpha}$ and $\tilde{\beta}$ are the unique harmonic representatives of $[\alpha]$ and $[\beta]$ and the inner product 
is the hermitian inner product we defined in \ref{hermitian_inner_product}.
\end{definition}

Let $or(E)$ be the set of oriented edges, we call an element $b\in or(E)$ a bond. Let $\overline{b}$ denote a reversal of orientation. Let $o(b)$ and $t(b)$ be the 
origin and terminal vertex of a bond $b$.

Let $A$ be an abelian group, the coefficients of the homology. Let $C_0(G,A)$ be the free $A$-module with generators in $V$. 
Let $C_1(G,A)$ be the $A$-module generated by 
$or(E)$ modulo the relation $\overline{b}=-b$.
The boundary map $\partial: C_1(G,A) \ra C_0(G,A)$ is defined by $\partial (b):=t(b)-o(b)$ and linearity. We then have $H_1(G,A)=ker(\partial)$.

We have the natural pairing $([\alpha], [p]) \mapsto \int_p \alpha$ for any $[\alpha] \in H^1(G,\R)$ and $[p] \in H_1(G,\R)$. 
This makes these two spaces dual to each other and induces an inner product on $H_1(G,\R)$.

\begin{lemma}
\label{inner_product}
 This inner product is equivalent to the one we get on $H_1(G,\R)$ as a subspace of $C_1(G,\R)$ with the inner 
product given by 
\begin{equation*}
e\cdot e' = \begin{cases}L(e) & e=e'\\
		 -L(e) & e=\overline{e'}\\
		0 & otherwise
\end{cases}
\end{equation*}
on edges and bilinear extension.
\end{lemma}

This might seem an awkward inner product if one thinks of vectors but the better analogy would be to think of characteristic functions of 
sets in $\R^n$ with an $L^2$ inner product.

\begin{remark}
The inner product plays well with our notion of edges of positive and negative overlap in Definition \ref{overlap}. The inner product 
of two cycles is equal to the difference between the length of the edges of positive and negative overlap. 
\end{remark}

\begin{definition}
 The {\bf Albanese torus} of a quantum graph is the Riemannian torus 
\begin{equation*}
 Alb(G):= H_1(G,\R) \slash H_1(G,\Z)
\end{equation*}
with inner product as in \ref{inner_product}.
The {\bf Jacobian torus} of a quantum graph is the Riemannian torus 
\begin{equation*}
 Jac(G):= H^1(G,\R) \slash H^1(G,\Z)
\end{equation*}
with inner product as in \ref{inner_prod_cohomology}.
Note that these are dual tori.
\end{definition}

\chapter{The Bloch spectrum}

In this chapter we will introduce the Bloch spectrum, first using differential forms and then using characters of 
the fundamental group. We show that the two notions are equivalent.

\begin{remark}
 The spectrum of the standard Laplacian determines the Euler characteristic via the trace formula \ref{trace_formula}. The multiplicity 
of the eigenvalue zero is equal to the number of connected components, $\dim H^0(G,\Z)$. Thus the spectrum of the standard Laplacian
 determines the dimension of $H^1(G,\Z)$.
\end{remark}

\section{The Bloch spectrum via differential forms}

\begin{proposition}
\label{1form_invariance}
 Let  $\alpha \in \Lambda^1$ and $\psi \in \Lambda^0$ be real and let $\beta=\alpha+d\psi$. Let $f$ be an 
eigenfunction of $\Delta_{\alpha}$ with eigenvalue $\lambda$. Then $e^{-2\pi i\psi}f$ is an eigenfunction of $\Delta_{\beta}$ with the same 
eigenvalue. That is, two operators whose $1$-forms differ by an exact $1$-form have the same spectrum.
\end{proposition}
\begin{proof}
We have
\begin{eqnarray*}
&& d_{\beta}^*d_{\beta}\left(e^{-2\pi i\psi}f\right)\\
&=& d_{\beta}^*\left( d(e^{-2\pi i\psi}f)+2\pi i e^{-2\pi i\psi}f\alpha + 2\pi i e^{-2\pi i\psi}fd\psi \right)\\
&=& d_{\beta}^*\left( e^{-2\pi i\psi}df+2\pi i e^{-2\pi i\psi}f\alpha \right)\\
&=& d_{\beta}^* \left(e^{-2\pi i\psi}d_{\alpha}f\right) \\
&=& d^*\left(e^{-2\pi i\psi}d_{\alpha}f\right) - 2\pi i\alpha(X_1) e^{-2\pi i\psi}d_{\alpha}f(X_1) - 2\pi i d\psi(X_1) e^{-2\pi i\psi}d_{\alpha}f(X_1)\\
&=& e^{-2\pi i\psi}d^*d_{\alpha}f -2\pi i\alpha(X_1) e^{-2\pi i\psi}d_{\alpha}f(X_1) \\
&=& e^{-2\pi i\psi}d_{\alpha}^*d_{\alpha}f
\end{eqnarray*}
Thus $f$ is an eigenfunction for $\Delta_{\alpha}$ if and only if $e^{-2\pi i\psi}f$ is an eigenfunction for $\Delta_{\beta}$ with the same eigenvalue.
\end{proof}

\begin{remark}
 Note that $Spec_{\alpha}(G)$ depends only on the coset of $[\alpha]$ in $H^1_{dR}(G,\R) \slash H^1_{dR}(G,\Z)$. 
\end{remark}

\begin{definition}
We define the {\bf Bloch spectrum} $Spec_{Bl}(G)$ of a quantum graph to be the map that associates to each $[\alpha]$ the 
spectrum $Spec_{\alpha}(G)$ where $[\alpha] \in H^1_{dR}(G,\R) \slash H^1_{dR}(G,\Z)$.

Note that we assume that we only know $H^1_{dR}(G,\R) \slash H^1_{dR}(G,\Z)$ as an abstract torus without any Riemannian structure.
\end{definition}

\begin{definition}
 We say that two quantum graphs $G$ and $G'$ are {\bf Bloch isospectral} if there is a Lie group isomorphism  
$\Phi: H^1_{dR}(G,\R)/H^1_{dR}(G,\Z)  \ra H^1_{dR}(G',\R)/ H^1_{dR}(G',\Z)$ 
such that $Spec_{\alpha}(G)=Spec_{\Phi(\alpha)}(G')$ for all $[\alpha] \in H^1_{dR}(G,\R)/H^1_{dR}(G,\Z)$.
\end{definition}

\begin{remark}
If $G$ is a tree its entire Bloch spectrum just consists of the spectrum of the standard Laplacian $\Delta_0$
and thus does not contain any additional information. 
\end{remark}

\section{The Bloch spectrum via characters of the fundamental group}

Let $\tilde{G}$ be the universal cover of $G$ and let $\pi_1(G)$ denote the fundamental group. Then $\pi_1(G)$ acts by deck 
transformations on $\tilde{G}$. Let $\chi: \pi_1(G) \ra \C^*$ be a character of $\pi_1(G)$.

We will study functions $\tilde{f}: \tilde{G} \ra \C$ that are continuous, satisfy Kirchhoff boundary conditions at the vertices, 
and that obey the transformation law
\begin{equation*}
\tilde{f}(\gamma x)=\chi(\gamma) \tilde{f}(x) 
\end{equation*}
for all $x\in \tilde{G}$ and $\gamma \in \pi_1(G)$. We refer to the space of these functions as $\Lambda^0_{\chi}(\tilde{G})$.

We associate to the character $\chi$ the spectrum of the standard Laplacian $d^*d$ on $\tilde{G}$ restricted to functions in 
$\Lambda^0_{\chi}(\tilde{G})$, we will denote it by $Spec(G,\chi)$.

\begin{definition}
We call the map that associates to each character $\chi$ of $\pi_1(G)$ the spectrum $Spec(G, \chi)$ the {\bf $\pi_1$-spectrum} of $G$.
\end{definition}

\section{Equivalence of the two definitions}

\begin{theorem}
The Bloch spectrum $Spec_{Bl}(G)$ and the $\pi_1$-spectrum of a quantum graph are equal.
There is a one-to-one correspondence $[\alpha] \mapsto \chi_{\alpha}$ between $H^1_{dR}(G,\R) \slash H^1_{dR}(G,\Z)$ and the set of characters 
of $\pi_1(G)$. It is given by 
\begin{equation*}
 \chi_{\alpha}(\gamma)=e^{-2\pi i\int_{\gamma}\alpha}
\end{equation*}
It induces the equality  $Spec(G, \chi_{\alpha})=Spec_{\alpha}(G)$.
\end{theorem}
\begin{proof}
The integral does not depend on either the representative in $\pi_1(G)$ nor on the one in $H^1_{dR}(G,\R)$ so this gives a well defined map. 
We also have $\chi_{\alpha}(\gamma_1 \cdot \gamma_2)=\chi_{\alpha}(\gamma_1)\chi_{\alpha}(\gamma_2)$ so this defines a character.

Let $f: G \ra \C$ and let $\tilde{f} :\tilde{G} \ra \C$ be the lift of $f$. Let $\tilde{\alpha}$ 
be the pullback of $\alpha$. 
As $H^1(\tilde{G})$ is trivial $\tilde{\alpha}$ is exact and there exists a function $\tilde{\varphi}: \tilde{G} \ra \C$ such that 
$\tilde{\alpha}=d\tilde{\varphi}$.

Let $\tilde{g}(x):=e^{-2\pi i\tilde{\varphi}(x)}\tilde{f}(x)$. We claim that $\tilde{g}$ is an eigenfunction in the $\pi_1$-spectrum if and only 
if $\Delta_{\alpha} f=\lambda f$. We need to show that $\tilde{g} \in \Lambda^0_{\chi_{\alpha}}(\tilde{G})$ and that 
$\Delta \tilde{g}= \lambda \tilde{g}$.

Let $\gamma \in \pi_1(G)$ and let $\tilde{\gamma}$ be the (unique) path in $\tilde{G}$ from $x$ to $\gamma x$. We have 
$\tilde{\varphi}(\gamma x)-\tilde{\varphi}(x)=\int_{\tilde{\gamma}}d\tilde{\varphi}$ 
by Stoke's theorem. So we get 
\begin{equation*}
 \tilde{g}(\gamma x)=e^{-2\pi i\tilde{\varphi}(\gamma x)}\tilde{f}(\gamma x)=
e^{-2\pi i\int_{\tilde{\gamma}}\tilde{\alpha}}e^{-2\pi i\tilde{\varphi}(x)}\tilde{f}(x)= \chi_{\alpha}(\gamma)\tilde{g}(x)
\end{equation*}
By Proposition \ref{1form_invariance} we have\\
\begin{equation*}
\Delta \tilde{g} = \Delta e^{-2\pi i\tilde{\varphi}}\tilde{f} = e^{-2\pi i\tilde{\varphi}}\Delta_{\tilde{\alpha}} \tilde{f}
\end{equation*}
Thus $\tilde{g}$ is an eigenfunction with eigenvalue $\lambda$ if and only if $f$ is.
\end{proof}

\begin{remark}
This theorem mirrors a similar result for tori, see \cite{Guillemin90}. 
\end{remark}

\chapter{The homology of a quantum graph}

In this chapter we will analyze the spectrum and the trace formula and extract information about the homology of the graph from it.

Before we state and prove the main theorem of this chapter we need a few definitions and a technical lemma.

\begin{definition}
We call a periodic orbit {\bf minimal} if it has minimal length within its homology class.
\end{definition}

\begin{remark}
Note that in general a given element in the homology might have more than one minimal periodic orbit that represents it.

On the other hand, all closed walks that contain no edge repetitions, and in particular all cycles are minimal. A cycle is 
also the unique minimal periodic orbit in its homology class.
\end{remark}

\begin{definition}
\label{generic}
We call a $1$-form $\alpha$ {\bf generic} if 
the image of the ray  $t\alpha$ in the torus $H^1_{dR}(G,\R) \slash H^1_{dR}(G,\Z)$ is dense. The $\alpha$'s with this property are dense. We pick and fix a single generic 
$\alpha$. 
\end{definition}

\begin{definition}
\label{frequencies}
To the fixed generic $\alpha$ we associate the following data.
\begin{enumerate}
 \item Let $\Psi$ be the linear map $ \Psi: H_1(G,\Z) \ra \R$ given by 
$[p] \mapsto 2\pi \int_p \alpha $. It associates to each periodic orbit its magnetic flux. 
\item We call the absolute values of the magnetic fluxes $\mu:=|\Psi([p])|=2\pi \left|\int_p \alpha\right|$ the 
{\bf frequencies} associated to $\alpha$.
\item We will denote the length of the minimal periodic orbit(s) associated to a frequency $\mu$ by $l(\mu)$.
\end{enumerate}
\end{definition}

\begin{remark}
The map $\Psi$ is two-to-one (except at zero) because we picked $\alpha$ to be generic.
The set of all frequencies $\mu$ union their negatives $-\mu$ and zero
forms a finitely generated free abelian subgroup of $\R$ that 
is isomorphic to  $H_1(G, \Z)$ via the map $\Psi$. 
\end{remark}

\begin{lemma}
\label{determine_cosine}
 Let $f$ be a function that is a linear combination of several cosine waves with different (positive) frequencies. 
\begin{equation*}
 f(t)=\sum_{j=1}^k\nu_j\cos(\mu_j t)
\end{equation*}
Then the values $f(t)$ for $t \in [0,\varepsilon)$ determine both $k$ and the individual frequencies $\mu_1, \ldots, \mu_k$.
\end{lemma}
\begin{proof}
 Assume without loss of generality that $0 < \mu_1 < \ldots < \mu_k$. We will show that we can determine $\mu_k$ and $\nu_k$ and then use 
induction.
We will look at the collection of derivatives of $f$ at $t=0$. We have
\begin{equation*}
 f^{(2n)}(0)=(-1)^n\sum_{j=1}^k\nu_j \mu_j^{2n}
\end{equation*}
There exists a unique number $\lambda >0$ such that\\
\begin{equation*}
 -\infty < \lim_{n \ra \infty} \frac{f^{(2n)}(0)}{(-\lambda)^n} < \infty \hspace{1cm} and \hspace{1cm} 
\lim_{n \ra \infty} \frac{f^{(2n)}(0)}{(-\lambda)^n} \neq 0
\end{equation*}
and we have $\lambda=\mu_k^2$ and $\lim_{n \ra \infty} \frac{f^{(2n)}(0)}{(-\lambda)^n}=\nu_k$. We can now look at the new function, 
\begin{equation*}
 \tilde{f}(t):=f(t)-\nu_k\cos(\mu_k t)
\end{equation*}
repeat the process, and determine $\mu_{k-1}$ and $\nu_{k-1}$. After finitely many steps we will end up with the constant function $0$.
\end{proof}

The following theorem is the key link between the Bloch spectrum and the quantum graph. All further theorems are just consequences 
of this one.

\begin{theorem}
\label{Bloch_H1}
 Given a generic $\alpha$, see Definition \ref{generic}, the part of the Bloch spectrum $Spec_{t\alpha}(G)$ for $t\in [0,\varepsilon)$ 
 determines the length of the minimal periodic orbit(s) of each element in $H_1(G, \Z)$. 
\end{theorem}

\begin{proof} 
We will show we can read off the set of frequencies $\mu$, see Definition \ref{frequencies}, associated to the generic $\alpha$ 
from the Bloch spectrum and determine the length $l(\mu)$ for each frequency.

We will look at the continuous family of $1$-forms $\alpha(t)=t\alpha$ and the associated 
operators $\Delta_{\alpha(t)}$ for our fixed generic $\alpha$ and $t\in [0,\varepsilon)$. If we plug the eigenvalues of these operators 
into the Fourier transform of the trace formula we get a family of distributions. Each of these distributions is a locally 
finite sum of Dirac-$\delta$-distributions (plus a constant term). The support of each of these $\delta$-distributions is the length of 
the periodic orbit(s) it is associated to and thus depends only on the underlying quantum graph and not on the $1$-form, see  
\ref{fourier_trace_formula}. 

Any periodic orbit $p$ that is homologically non-trivial has a corresponding partner $\overline{p}$ which is the 
same closed walk with opposite orientation. Their coefficients are related by 
$\mathcal{A}_p(\alpha)=\overline{\mathcal{A}_{\overline{p}}(\alpha)}$ 
as the vertex scattering coefficients and the length are the same and the magnetic flux changes sign. Thus for each such pair 
we would observe a factor of the form $2{\rm Re} \mathcal{A}_p(t\alpha)$ in the Fourier 
transform of the trace formulae for $\Delta_{\alpha(t)}$. We have 
\begin{equation*}
 2{\rm Re} \mathcal{A}_p(t)=2{\rm Re} \left( \tilde{l}_pe^{2\pi i t\int_p\alpha}\prod_{b\in p} \sigma_{t(b)}\right) =\nu \cos\left(2\pi \int_p\alpha t\right)
\end{equation*}
by Theorem \ref{trace_formula} where $\nu=2\tilde{l}_p\prod_{b\in p} \sigma_{t(b)}$. So for each such pair of periodic orbits 
there is a magnetic flux $\Psi([p])=2\pi \int_p\alpha$ 
 and a factor $\nu$ that is always nonzero. Moreover the factor $\nu$  is positive if the periodic orbit 
contains no backtracks. As the magnetic flux appears in a cosine wave we can only know its absolute value, that is, 
the frequency, see Definition \ref{frequencies}. 

Pick a length of periodic orbits $l$. 
If we look at the family of $1$-forms $\alpha(t)$ we get a continuous family of coefficients $\mathcal{A}^l(t)=\sum_{l_p=l}\mathcal{A}_p(t)$.
As we did not make any assumptions on the underlying quantum graph there can be 
multiple periodic orbits with the same length.
Thus each coefficient $\mathcal{A}^l(t)$ is a linear combination 
of a constant term and several cosine waves with 
different frequencies. The constant part comes from homologically trivial periodic orbits of length $l$. The cosine waves correspond to the 
homologically nontrivial periodic orbits of length $l$. We can now apply Lemma \ref{determine_cosine} to the function  $\mathcal{A}^l(t)$ 
and read off all the frequencies occurring at that length.

As we go through the different lengths in the spectra starting at zero we will pick up a collection of different frequencies. Each 
frequency will appear multiple times at different lengths since there are multiple periodic orbits that represent the same element in the homology and 
thus have the same frequency.

The frequency corresponding to a particular element in $H_1(G, \Z)$ can only be realized by periodic orbits that 
represent this element in the homology because we picked $\alpha$ to be generic, see \ref{generic}. Going through the lengths 
starting at zero this frequency can appear at the earliest at 
the length of the corresponding minimal periodic orbit(s). The minimal periodic orbits need not be unique but as they are minimal 
they contain no backtracks. 
Thus their $\nu$ coefficients are all strictly bigger than $0$, so their sum cannot vanish and the frequency will indeed 
appear in the coefficient $\mathcal{A}^l(t)$ at the minimal length. 
This gives us the length $l(\mu)$ associated to each frequency $\mu$. 
\end{proof}

\begin{remark}
As we picked $\alpha$ to be generic, the maximal number of frequencies that are linearly independet over $\Q$ is 
equal to $\dim H_1(G,\Z)$. 
Thus we can observe from the number of rationally independent frequencies whether an arbitrary $\alpha$ is generic or not.
\end{remark}

\begin{remark}
Without any genericity assumptions on the edge lengths in the quantum graph it can happen that there are multiple non-minimal periodic orbits that are homologous and of 
the same length. We would not be able to distinguish them directly in the trace formula, it can even happen that their $\nu$-coefficients cancel out 
and we would not observe them at all. 
\end{remark}

\chapter{Determining graph properties from the Bloch spectrum}

We will now use the information gained in the last chapter and translate it into graph properties that are determined by the Bloch spectrum.

\section{The Albanese torus}

\begin{lemma}
\label{recognize_cycle}
Given a frequency $\mu$ the following two statements are equivalent:
\begin{enumerate}
\item The minimal periodic orbit associated to $\mu$ is a cycle in the graph.
 \item  There are no two frequencies $\kappa$, $\kappa'$, $\mu= |\kappa \pm \kappa'|$ with the property that 
$l(\kappa)+l(\kappa') \le l(\mu)$.
\end{enumerate}
\end{lemma}
\begin{proof}
 We will prove both directions by contradiction. 

Assume the minimal periodic orbit associated to $\mu$ is not a cycle, then it has to go through some vertex at 
least twice. Thus we can separate the periodic orbit into two shorter periodic orbits. Let $\kappa$ and $\kappa'$ be the  frequencies associated to the two pieces. 
Then $\mu=|\kappa \pm \kappa'|$ and because the two pieces are not necessarily minimal $l(\kappa)+l(\kappa') \le l(\mu)$.

Conversely, suppose $\mu$ admits a decomposition $\mu=|\kappa \pm \kappa'|$. Let $c_{\mu}$, $c_{\kappa}$ and $c_{\kappa'}$ 
denote the minimal periodic orbits associated to the frequencies.  
If we have $l(\kappa)+l(\kappa') = l(\mu)$ then $c_{\mu}=c_{\kappa}\cup c_{\kappa'}$ and $c_{\kappa}$ and $c_{\kappa'}$ must have a vertex in common so $c_{\mu}$ is not a 
cycle. If we have $l(\kappa)+l(\kappa') < l(\mu)$ then
 the periodic orbits $c_{\kappa}$ and $c_{\kappa'}$ are disjoint and $c_{\mu}$ realizes the connection between them so it uses the edges 
between them twice and is not a cycle.
\end{proof}

\begin{remark}
\label{check_overlap_sign}
 Let $\mu_1$ and $\mu_2$ be two frequencies such that the associated minimal periodic orbits $c_1$ and $c_2$ are cycles. 
These cycles have an orientation induced from the $1$-form $\alpha$. The frequency $\mu_1+\mu_2$ corresponds to the pair of periodic orbits 
that is homologous to $c_1\cup c_2$ and $-(c_1\cup c_2)$. Thus if $l(\mu_1+\mu_2) < l(\mu_1)+l(\mu_2)$ 
then $c_1$ and $c_2$ have edges of negative overlap. The frequency $|\mu_1-\mu_2|$ corresponds to the pair of periodic orbits 
that is homologous to $c_1\cup (-c_2)$ and $(-c_1)\cup c_2$. Thus if $l(\mu_1-\mu_2) < l(\mu_1)+l(\mu_2)$ 
then $c_1$ and $c_2$ have edges of positive overlap. 
\end{remark}

\begin{theorem}
\label{determine_Albanese}
The Bloch spectrum of $G$ determines the Albanese torus $Alb(G)$ as a Riemannian manifold.
\end{theorem}
\begin{proof}
Pick a minimal set of generators $\mu_1, \ldots, \mu_n$ of the group spanned by the frequencies such that the associated minimal periodic orbits 
are all cycles. Such a basis exists by Lemma \ref{basis_of_cycles}. 
Associate to them a set of 
vectors $v_1, \ldots, v_n$ satisfying $|v_i|^2:=l(\mu_i)$ and $2\langle v_i,v_j \rangle := l(\mu_i+\mu_j)-l(|\mu_i-\mu_j|)$ for all $i \neq j$. 
This uniquely determines a torus with spanning vectors $v_1, \ldots, v_n$. 

If the cycles associated to $\mu_i$ and $\mu_j$ share no edges we have $l(\mu_i+\mu_j)=l(|\mu_i-\mu_j|)\ge l(\mu_i)+l(\mu_j)$ so the associated vectors 
 are orthogonal. 

If the cycles associated to $\mu_i$ 
and $\mu_j$ share edges the length $l(\mu_i+\mu_j)$ is twice the length of all edges of positive overlap plus the length of all 
edges that are part of one cycle but not the other. The length $l(|\mu_i-\mu_j|)$ is twice the length of all edges of negative overlap 
plus the length of all edges that are part of one cycle but not the other. Thus $l(\mu_i+\mu_j)-l(|\mu_i-\mu_j|)$ is twice the difference 
of the length of edges of positive overlap and the length of edges of negative overlap.

Therefore the torus is isomorphic to the Albanese torus of the quantum graph by Lemma \ref{inner_product}. 
\end{proof}

The complexity of a graph is the number of spanning trees.
\begin{corollary}
 If the quantum graph is equilateral the Bloch spectrum determines the complexity of the graph.
\end{corollary}
\begin{proof}
This follows directly from a theorem in \cite{KotaniSunada00}. For combinatorial graphs the complexity of the graph is given by 
$K(G)=\sqrt{vol(Alb(G))}$. The Albanese torus of an equilateral quantum graph is identical to the Albanese torus of 
the underlying combinatorial graph.
\end{proof}

\begin{remark}
Leaves in a graph are invisible to the homology. So it is not clear whether the entire Bloch spectrum gives us any more information about them 
than the spectrum of a single Schr\"odinger type operator. There are examples of trees that are isospectral for the standard Laplacian,
 see for example \cite{GutkinSmilansky01}. 
\end{remark}

\begin{remark}
The Albanese torus distinguishes the isospectral examples of van Below in \cite{vonBelow99}. Thus the spectrum of a single Schr\"odinger type 
operator does not determine the Albanese torus. In one of the two graphs  
two periodic orbits of length $3$ can be composed to get a periodic orbit of length $4$. Thus the lattice that corresponds to the Albanese 
torus contains two vectors of length $3$ whose sum has length $4$. In the other graph this is not the case. In particular these two graphs are 
not Bloch isospectral by Theorem \ref{determine_Albanese}.
\end{remark}

\section{The block structure}

\begin{theorem}
\label{block_structure}
The Bloch spectrum of a leafless quantum graph determines its block structure (see Definition \ref{define_block}).
It also determines the dimension of the homology of each block.
\end{theorem}
\begin{proof}
 Pick a minimal set of generators $\mu_1, \ldots, \mu_n$ of the group spanned by the frequencies such that the associated minimal 
periodic orbits are all cycles.
 A cycle is necessarily contained within a single block, see \ref{cycle_single_block}. Declare two generators 
 equivalent if the associated cycles share edges regardless of orientation. This generates an equivalence relation. 
Let $\mathfrak{B}$ be the set of equivalence classes, it corresponds to the set of blocks of $G$, see \ref{define_block}. 
The number of generators in each equivalence class is the dimension of the homology of that block.

Let $B_1, B_2 \in \mathfrak{B}$. Let $\{\mu^j_i\}_i$ be the subset of frequencies that is $B_j$, $j=1,2$. Then we can find the distance 
between the two blocks by computing 
\begin{equation*}
 d(B_1,B_2):=\frac{1}{2}\min_{i,i'}\left( l(\mu^1_i + \mu^2_{i'})-l(\mu^1_i)-l(\mu^2_{i'}) \right)
\end{equation*}
That is we compute the distance between any basis cycle in one 
block to any basis cycle in the other and minimize over all pairs of basis cycles in the blocks. This distance is greater equal zero 
and equal to zero if and only if the blocks share a vertex.

We will now set up a situation where we can apply Lemma \ref{tree_from_leaves}. To do so we need to find out which blocks are leaves in 
the block structure and which ones are inner vertices. We will then cut the block structure into smaller pieces such that 
all blocks are leaves in the smaller pieces.

Whenever we have a triple of blocks satisfying $d(B_2,B_3) > d(B_1,B_2)+d(B_1,B_3)$, 
that is, a failure of the triangle inequality, 
we know that $B_1$ has to be an inner vertex in the block structure of $G$. The path between the blocks $B_2$ and $B_3$ has to pass through $B_1$ 
and use some edges within the block $B_1$. Once we have identified a block, say $B_1$, as an inner block we can separate the remaining 
blocks into groups depending on where the path from the block to $B_1$ is attached on $B_1$. If $d(B_i,B_j) > d(B_1,B_i)+d(B_1,B_j)$ then 
the paths from $B_1$ to $B_i$ and $B_j$ 
are attached 
at different cut vertices of $B_1$, if $d(B_i,B_j) \le d(B_1,B_i)+d(B_1,B_j)$ they are attached at the same cut vertex. 
Within each of these groups the block $B_1$ is a leaf in the block structure. 

Thus we have cut the initial block structure into several smaller pieces each of them 
including $B_1$ and $B_1$ is a leaf in each of them.
 We can repeat this process of identifying an inner block and cutting the block structure into smaller pieces on each of these pieces until 
all the pieces have no inner block vertices. This reduces the 
problem to recovering the block structure of graphs where all blocks are leaves. 

All remaining inner vertices have to be vertices of the initial graph $G$ and thus have degree at least $3$. As $G$ is leafless 
all leaves in the block structure are fat vertices. Hence we can recover the block 
structure of each of the smaller pieces by using Lemma \ref{tree_from_leaves}. We can then find the block structure of the entire graph 
by gluing the pieces together at the inner blocks.
\end{proof}

\section{Planarity and dual graphs}

\begin{theorem}
\label{planarity}
The Bloch spectrum determines whether or not a graph is planar.
\end{theorem}
\begin{proof}
The homology admits infinitely many bases, but a graph has only finitely many cycles and thus there are only finitely many 
bases consisting of cycles. Thus there are only finitely many minimal sets of generators $\mu_1, \ldots, \mu_n$ of the group spanned by 
the frequencies 
such that the minimal periodic orbits associated to them are all cycles. Given such a basis we can choose for each generator to either 
keep the orientation induced by $\alpha$ or choose the reverse orientation. Denote the basis 
elements with a choice of orientation 
by $\gamma_1=\pm \mu_1, \ldots, \gamma_n=\pm \mu_n$. For any pair of oriented basis elements $\gamma_i$ and $\gamma_j$ we can check 
whether the associated cycles have edges of positive overlap by checking whether 
$l(|\gamma_i+\gamma_j|)-l(|\gamma_i|)-l(|\gamma_j|)>0$, see \ref{check_overlap_sign}. Thus we can check whether the 
$\gamma_1, \ldots, \gamma_n$ correspond to a basis of the homology that consists of oriented cycles having no positive overlap.
The graph is planar if and only if we can find such a basis by Theorem \ref{simple=planar}. 
\end{proof}

\begin{remark}
 Planarity is a property that is not determined by the spectrum of a single Schr\"odinger type operator. There is an example of two isospectral 
quantum graphs in \cite{vonBelow99} where one is planar and the other one is not.
\end{remark}

If the graph is planar we will fix a non-positive basis $\gamma_1, \ldots, \gamma_n$ (see Definition \ref{non-positive_basis}) coming 
from the frequencies $\mu_1, \ldots, \mu_n$.
We know that the basis elements are the boundaries of the inner faces in a suitable embedding of the graph by Lemma 
\ref{positive=facial}. We will use this fact to construct an abstract dual of the graph.

\begin{theorem}
\label{abstract_dual}
 The Bloch spectrum of a planar, 2-connected graph determines a dual of the graph. Thus the Bloch spectrum determines 
planar, 2-connected graphs up to 2-isomorphism (see Lemma \ref{2-isomorphic}).
\end{theorem}

Before we show this we need two lemmata.

\begin{lemma}
Let $G$ be planar and $2$-connected. In the embedding where the $\gamma_1, \ldots, \gamma_n$ are the boundaries of the inner faces the 
boundary of the outer face is given by 
\begin{equation*}
\gamma_0:=-\sum_{l=1}^{n}\gamma_l      
\end{equation*}
The sign orients it so that it does not have edges of positive overlap with any of the $\gamma_l$.
\end{lemma}

The boundary of the outer face $\gamma_0$ is a cycle in the graph because $G$ is $2$-connected.

\begin{lemma}
\label{decomposition}
Let $G$ be planar and $2$-connected. 
Then we can determine the number of edges that any two of the cycles $\gamma_0, \ldots, \gamma_n$ have in common.
\end{lemma}
\begin{proof}
Recall that $\mu_i=|\gamma_i|$. If $l(\mu_i+\mu_j) \ge l(\mu_i)+l(\mu_j)$ then $\gamma_i$ and $\gamma_j$ share no edges. 
We will assume from now on that $l(\mu_i+\mu_j) < l(\mu_i)+l(\mu_j)$. 
Suppose $\gamma_i$ and $\gamma_j$ share $k$ 
edges or single vertices. 

\begin{figure}[ht]
\centering
\scalebox{0.9}{\includegraphics{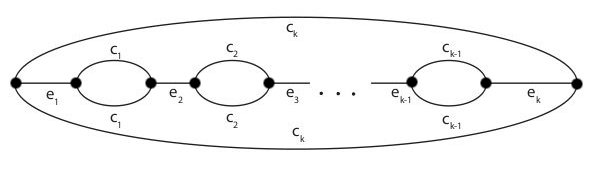}}
\caption{The graph with the two cycles $\gamma_i$ and $\gamma_j$}
\label{determine_edge_count}
\end{figure}

Figure \ref{determine_edge_count} shows the graph $G$. Here $\gamma_i$ and $\gamma_j$ bound the two big faces and the $e_l$ are 
the edges or single vertices these two cycles share. The remainder 
of the graph is contained inside the small cycles labeled $c_l$, $l=1,\ldots ,k-1$ and outside the big cycle $c_k$.

We can decompose the minimal periodic orbit associated to the frequency $\mu_i+\mu_j$ into several cycles $c_1, \ldots, c_k$ by applying 
Lemma \ref{recognize_cycle} repeatedly. These cycles do not share any edges, they 
can share vertices. If the decomposition yields $k$ cycles, then $\gamma_i$ and $\gamma_j$ have $k$ distinct 
components in common. Each of these components is either a single edge or a vertex. Whenever it is a vertex that means that two of the $c_l$ have 
this vertex in common and thus have distance zero from each other. As we can check for any pair $c_l$ and $c_{l'}$ whether 
$l(c_l+c_{l'})=l(c_l)+l(c_{l'})$ we can find all instances where this happens. All remaining common components then must correspond to a common
edge of $\gamma_i$ and $\gamma_j$.
\end{proof}

\begin{remark}
 Lemma \ref{decomposition} is false without the planarity assumption. There exist two cycles $\gamma_1, \gamma_2$ in $K_{3,3}$ that 
share $3$ edges but $\gamma_1 \cup \gamma_2$ is homologous to a single cycle. These two cycles have no edges of positive overlap.
\end{remark}

\begin{proof}
 of Theorem \ref{abstract_dual}:\\
The cycles $\gamma_0, \ldots, \gamma_n$ are the set of all boundaries of faces in a suitable embedding of the graph. Therefore they are 
the vertices of a geometric dual. By 
Lemma \ref{decomposition} we know the number of edges any two of these faces have in common, which corresponds to the number of edges 
between the two vertices in the geometric dual.
\end{proof}

The particular geometric dual we get from this process depends on the non-positive basis we have chosen.

\begin{corollary}
\label{determine_3con_combo}
 The Bloch spectrum identifies and determines planar, $3$-connected graphs combinatorially. 
\end{corollary}
\begin{proof}
Note that $3$-connected implies $2$-connected, see Definition \ref{k-connected}.
A graph is $2$-connected if its block structure 
consists of a single fat vertex. We have shown that we can identify planar graphs in Theorem \ref{planarity}. 
We found a geometric dual of a $2$-connected planar graph in Theorem \ref{abstract_dual}. If the dual of the dual is 
$3$-connected it will be unique and therefore isomorphic to the original graph. 
\end{proof}

\chapter{Determining the edge lengths}

In this chapter we will show that we can recover all the edge lengths of a $3$-connected planar graph if we know the underlying 
combinatorial graph.

The Bloch spectrum only gives us a map from the abstract torus $H^1(G,\R) \slash H^1(G,\Z)$ to the spectra.
By Theorem \ref{Bloch_H1} we know the length of the minimal periodic orbit(s) associated to each element in $H_1(G,\Z)$. 
Here we need a little more, we 
want to associate the lengths we get from the Bloch spectrum with the periodic orbits in the combinatorial graph. 

When we construct a dual graph in Theorem \ref{abstract_dual} we can keep track of the lengths associated to the 
minimal set of generators $\mu_1, \ldots, \mu_n$ of the group spanned by the frequencies. The vertices of the 
dual graph correspond to these frequencies. The vertices of the dual graph then correspond to a set of cycles in the dual of the dual 
that generates the homology. If the graph is $3$-connected the dual of the dual is isomorphic to the original graph so we can associate 
the frequencies and their lengths to the periodic orbits in the graph.

\begin{theorem}
\label{edge_length}
The Bloch spectrum identifies and completely determines $3$-connected planar quantum graphs.
\end{theorem}
\begin{proof}
We have already shown in Corollary \ref{determine_3con_combo} that the Bloch spectrum identifies $3$-connected planar quantum graphs 
and determines their underlying combinatorial graph. All that remains to be shown is that we can determine all the edge lengths.
By the remarks above the Bloch spectrum determines the length of all cycles in the graph. We can now apply Lemma 
\ref{determine_edge_length} to determine the lengths of all the edges.
\end{proof}

\begin{remark}
 One can show that given the Bloch spectrum and the underlying combinatorial graph of an arbitrary quantum graph one can associate 
the frequencies from the Bloch spectrum and their lengths to the closed walks in the combinatorial graph. However, using this 
information to determine all the individual edge lengths is more delicate.
\end{remark}

\chapter{Disconnected graphs}

\begin{remark}
 When we bounded the size of isospectral families in Theorem \ref{isospectral_bound} we estimated the total number of 
combinatorial graphs with a
given number of edges. This estimate counts combinatorial graphs independent of whether there are connected or not. Therefore
the bound given in this theorem is valid for all quantum graphs, connected or not.
\end{remark}

If we do not assume that the quantum graph is connected we get a component-wise version of Theorem \ref{Bloch_H1}.

\begin{proposition}
Let $G$ be a quantum graph that may or may not be connected. Then the spectrum of the standard Laplacian $\Delta_0$ determines 
the number of connected components. Denote the connected components by $G_1, \ldots, G_k$. 

Given a generic $\alpha$, see Definition \ref{generic}, the part of the Bloch spectrum $Spec_{t\alpha}(G)$ for $t\in [0,\varepsilon)$ 
 determines the groups $H_1(G_i, \Z)$ and the length of the minimal periodic orbit(s) of each element in $H_1(G_i, \Z)$ for 
each component $i=1, \ldots, k$. 
\end{proposition}
\begin{proof}
The multiplicity of the eigenvalue $0$ of the standard Laplacian $\Delta_0$ is equal to the number of connected components.

The trace formula in Theorem \ref{trace_formula} still holds for disconnected graphs. The two sums over eigenvalues and periodic 
orbits are just unions over the connected components. The total length of the quantum graph is additive and the Euler characteristic 
is well defined for disconnected graphs, too.

Thus we can copy most of the proof of Theorem \ref{Bloch_H1} verbatim and read out a set of frequencies from the Bloch spectrum. 
Every frequency we get is associated to a single periodic orbit that belongs to only one of the connected components.
If the sum of two frequencies is a frequency, then these two frequencies belong 
to the same connected component of $G$. If it is not they belong to different connected components.
Thus the set of frequencies (union their negatives and zero) will not form one finitely generated free abelian subgroup of $\R$ that 
is isomorphic to $H_1(G,\Z)$. Instead it will form $k$ disjoint (apart from zero) finitely generated free abelian subgroups 
of $\R$ that are isomorphic to the $H_1(G_i,\Z)$ for $i=1, \ldots, k$. 

We can now assign a length to each frequency the same way as in Theorem \ref{Bloch_H1}.
\end{proof}

As all our subsequent theorems are just consequences of Theorem \ref{Bloch_H1} they also hold component-wise.

\begin{corollary}
 Theorems \ref{determine_Albanese}, \ref{block_structure}, \ref{planarity}, \ref{abstract_dual} and \ref{edge_length} all hold 
component-wise.
\end{corollary}

\backmatter

\bibliographystyle{amsalpha}
 \appendix
\addcontentsline{toc}{chapter}{Bibliography}
\bibliography{/home/ralf/personalfiles/phd/literatur}



\end{document}